\documentclass[a4paper, 11pt]{article}
\usepackage{geometry}
\usepackage[utf8]{inputenc}
\usepackage[english]{babel}
\usepackage{amsmath,amssymb,amsthm,graphicx,mathrsfs,url}

\usepackage{authblk}

\usepackage{dsfont} 
\usepackage{amsmath,amssymb,amsthm}
\usepackage[justification=centering]{caption}
\usepackage[colorlinks=true, citecolor=green, linkcolor=blue]{hyperref}
\usepackage[nameinlink, capitalize]{cleveref}

\usepackage{marginnote}
\usepackage[colorinlistoftodos]{todonotes}
\usepackage{varwidth}

\usepackage{color}
\usepackage{float}
\usepackage[super]{nth}
\usepackage{booktabs}
\usepackage{rotating}
\usepackage{verbatim}
\usepackage{multirow}
\usepackage{physics}
\usepackage{mathtools}
\usepackage{subcaption}
\usepackage{bm}
\usepackage[toc,page]{appendix}
\usepackage[shortlabels]{enumitem}
\usepackage{nomencl}

\numberwithin{equation}{section}

\newcounter{mnotecount}[section]

\theoremstyle{plain}
\newtheorem{theorem}{Theorem}
\newtheorem*{theorem*}{Theorem}
\newtheorem{corollary}{Corollary}
\newtheorem{lemma}{Lemma}
\newtheorem{proposition}{Proposition}

\theoremstyle{definition}

\newtheorem{definition}{Definition}
\newtheorem{remark}{Remark}

\usepackage{yfonts}

\usepackage[colorinlistoftodos]{todonotes}

\begin{document}

\title{The Semi-Classical Limit from the Dirac Equation with Time-Dependent External Electromagnetic Field to Relativistic Vlasov Equations
}
\author[a]{François Golse}
\author[b]{Nikolai Leopold}
\author[c,d]{Norbert J. Mauser}
\author[a,c,d]{Jakob Möller}
\author[e,f]{Chiara Saffirio}
\affil[a]{CMLS, École Polytechnique, F-91128 Palaiseau}
\affil[b]{Constructor Univ., Bremen}
\affil[c]{Research Platform MMM "Mathematics-Magnetism-Materials" c/o Fak. Math., Univ. Wien, A-1090 Vienna}
\affil[d]{Wolfgang Pauli Institut, A-1090 Vienna}
\affil[e]{Albert-Ludwigs-Univ. Freiburg}
\affil[f]{Univ. Basel}
\maketitle
\begin{abstract}
We prove the mathematically rigorous (semi-)classical limit  
$\hbar  \to 0$ 
of the Dirac equation with time-dependent external electromagnetic field to relativistic Vlasov equations with Lorentz force for electrons and positrons. In this limit antimatter and spin remain as intrinsically relativistic effects on a classical level. Our global-in-time results use Wigner transforms and a Lagrange multiplier viewpoint of the matrix-valued Wigner equation. In particular, we pass to the limit in the ``full" Wigner matrix equation without projecting on the eigenspaces of the matrix-valued symbol of the Dirac operator. In the limit, the Lagrange multiplier maintains the constraint that the Wigner measure and the symbol of the Dirac operator commute and vanishes when projected on the electron or positron eigenspace. This is a different approach to the problem as discussed in [P. Gérard, P. Markowich, N.J. Mauser, F. Poupaud: Comm. Pure Appl. Math. 50(4):323–379, 1997], where the limit is taken in the projected Wigner equation. By explicit calculation of the remainder term in the expansion of the Moyal product we are able to generalize to time-dependent potentials with much less regularity. We use uniform $L^2$ bounds for the Wigner transform, which are only possible for a special class of mixed states as initial data. \\


\end{abstract}


\section{Introduction}

In relativistic quantum electrodynamics a fast electric charge $q$ of spin $\tfrac{1}{2}$ (fermion) is described by the Dirac equation for a 4-spinor $\Psi^{\hbar} \in L^2(\mathbb{R}^3, \mathbb{C}^4)$ containing an electron and a positron component with two spin directions each. The Dirac equation is given by
\begin{equation}
    i\hbar \partial_t \Psi^{\hbar} =   H \Psi^{\hbar} :=  \left(\alpha \cdot (-i\hbar \nabla - \frac{q}{c}A) + \beta - q A_0\right)\Psi^{\hbar},
    \label{eq:dirac equation spinor}
\end{equation}
where 
$A_0\in \mathbb{R}$ is the scalar electric potential and $A \in \mathbb{R}^3$ is the magnetic vector potential. The $4\times 4$ matrices $\alpha, \beta$ are the Dirac matrices, defined below in \eqref{eq:dirac matrices}.  

We explicitly denote the dependence of $\Psi^{\hbar}$ on the Planck constant $\hbar$ by a superscript since we are interested in the $\hbar$ asymptotics\footnote{$\hbar\approx 1.055 \cdot 10^{-34} \mathrm{m^2 kg / s}$ appears as small in an appropriate semiclassical scaling.} and we later rescale the equation such that $\hbar = \varepsilon$ is a small parameter. 

The electromagnetic potential $(A_0,A)$ in the matrix-valued Dirac Hamilton operator $H$ can be given externally or be computed self-consistently, yielding the nonlinear Dirac-Maxwell equation \cite{thaller1992dirac}.
Among many related mathematical questions we mention the asymptotic analysis for ``unbounded speed of light" $c \to \infty$, i.e. the ``non-relativistic limit" (or ``Post-Newtonian limit") towards Schrödinger- and Pauli equations \cite{bechouche1998semi,itzykson2012quantum}, which were justified rigorously even for the nonlinear Dirac-Maxwell equation in \cite{bechouche2005asymptotic, masmoudi2003nonrelativistic}).

In this work we deal with another important asymptotic of the Dirac equation, the ``(semi-)classical limit"  
of vanishing Planck constant  $\hbar \to 0$ towards fully relativistic classial physics as described by relativistic Vlasov (Liouville) equations with Lorentz force for phase space density functions. Wigner methods allow for such global in time limits towards kinetic equations (the other main technique would be WKB methods for small time results towards fluid-type equations).

For the linear Dirac equation with time-independent external electromagnetic potentials the first result for the global-in-time semiclassical limit was given in the seminal paper of Gérard, Markowich, Mauser, Poupaud \cite{gerard1997homogenization} where ``Wigner matrices" were introduced. In the remarkable subsequent paper of Spohn \cite{spohn2000semiclassical}, these Wigner transform techniques were used to derive – by taking the semiclassical limit of the Dirac equation – models like the {BMT equation} which are classical equations for the spinning electron. Let us mention that Spohn found a small mistake in \cite{gerard1997homogenization}, corrected in \cite{gerard2000erratum}.

For the nonlinear Dirac-Maxwell system, the only result for the semiclassical limit is the work \cite{mauser2007convergence} where the simultaneous non-relativistic and semiclassical limit to the non-relativistic Vlasov-Poisson equation is somewhat simplified to a relativistic perturbation of non-relativistic quantum mechanics (uniform estimates in $\hbar$ were obtained in \cite{bechouche2005asymptotic}, but not in \cite{masmoudi2003nonrelativistic}). \\

The goals of this paper are, on the one hand, to significantly advance from the case of smooth, time-independent potentials with polynomial growth at infinity as in \cite{gerard1997homogenization, spohn2000semiclassical} (using general methods of pseudo-differential calculus à la Hörmander) to low regularity, time-dependent potentials, i.e. non-autonomous Hamiltonians, which we achieve by explicit computation of the remainder term in the Wigner equation. On the other hand, we contribute to a better mathematical understanding of the classical, relativistic description of spin and antimatter, which remain when quantum mechanics have vanished with $\hbar$.

The regularity for the potentials used in this paper is the natural setting for the definition of trajectories for the relativistic Vlasov-Maxwell equation, the derivation of which from a quantum mechanical model such as the Dirac-Maxwell equation remains a hard open problem.  
\\

In the limit $\hbar \rightarrow 0$ , the dynamics of the electron ($+$) and positron ($-$) are governed by the classical relativistic Hamilton functions $H_{\pm}$ given by 
\begin{equation}
    H_{\pm} = \pm \sqrt{1+|\xi-A|^2} -A_0.
    \label{eq:classical hamiltonian}
\end{equation}
At first glance, spin does not appear in \eqref{eq:classical hamiltonian}. The spin – postulated by Goudsmit and Uhlenbeck in 1926 \cite{uhlenbeck1925ersetzung} in order to explain the anomalous Zeeman effect – occurs intrinsically in the Dirac equation (which Dirac derived two years later in 1928 \cite{dirac1928quantum}) and therefore one might be tempted to consider it as a quantum mechanical phenomenon that disappears in the (semi-)classical limit. However, it was discovered by Bargmann, Michel and Telegdi \cite{bargmann1959precession}, based on earlier works by Thomas, Frenkel and Kramers \cite{thomas1926motion, frenkel1926elektrodynamik, kramers1964quantum}, that the equation for the precession of the ``classical intrinsic angular momentum", now known as \emph{BMT equation}, can be derived solely from the principles of special relativity and the postulate of an intrinsic angular momentum as predicted in \cite{bargmann1959precession}, without the need to refer to quantum mechanics and the Dirac equation. Spohn in \cite{spohn2000semiclassical} showed that the classical BMT equation for the classical spin precession can indeed be derived from the Dirac equation and the quantum mechanical spin precession.

In this work we extend and also shed some new light on the semiclassical limit of the Dirac equation as it was treated in \cite{gerard1997homogenization, spohn2000semiclassical}. The framework of \cite{gerard1997homogenization, spohn2000semiclassical} introduces matrix-valued Wigner transforms in order to arrive at a phase space description of \eqref{eq:dirac equation spinor}. The classical motion described by the Liouville equation with the relativistic Hamilton functions \eqref{eq:classical hamiltonian} is then derived by projecting on the eigenspaces of the symbol of the Dirac Hamilton operator in \eqref{eq:dirac equation spinor}, the eigenvalues of which are given by \eqref{eq:classical hamiltonian}, and then passing to the limit. Somewhat differently, we pass to the limit without projecting first, thereby obtaining the \emph{matrix-valued Vlasov (Liouville) equation} for the Wigner matrix measure $W(t,x,\xi)$, 
\begin{align}
    {}&\partial_t W = -i[P,Y] + \frac{1}{2}(\{P,W\}-\{W,P\}), \label{eq:matrix valued liouville introduction}\\
    &[P,W] = 0,\label{eq:commutator constraint}
\end{align}
where $P = P(t,x,\xi)$ is the symbol \eqref{eq: time dependent electromagnetic Dirac symbol} of the Dirac Hamilton operator $H$ and where $[\,\cdot\,,\cdot \,]$ denotes the commutator of two $m\times m$-matrix-valued functions on phase space $\mathbb{R}^d_x \times \mathbb{R}^d_{\xi}$ and $\{\,\cdot\,,\cdot \,\}$ denotes the Poisson bracket on $\mathbb{R}^d_x \times \mathbb{R}^d_{\xi}$.
This equation contains all the information about the electron and positron at once, including the information about the spin. This way we obtain a direct correspondence between \eqref{eq:dirac equation spinor} and the classical dynamics \eqref{eq:matrix valued liouville introduction} of the full Wigner matrix measure $W$.

The matrix-valued Liouville equation \eqref{eq:matrix valued liouville introduction} for $W$ carries an additional term $-i[P,Y]$ compared to the projected Liouville equation obtained in \cite{gerard1997homogenization}. Here, $Y$ can be thought of as the Lagrange multiplier maintaining the constraint \eqref{eq:commutator constraint} that the Wigner matrix measure $W$ commutes with the symbol $P$ of the Dirac Hamiltonian $H$. 
Once projected on the eigenspaces $P$, the Lagrange multiplier disappears and we retrieve the situation described in \cite{gerard1997homogenization, spohn2000semiclassical}. The approach to use a Lagrange multiplier is inspired by an analogy with the incompressible Navier-Stokes equation, where the pressure serves as the Lagrange multiplier maintaining the incompressibility constraint. We explain this analogy in greater detail in Section \ref{section analogy navier stokes}.

We can formulate the following informal version of our main theorems: 
\begin{theorem*}
Let $\varepsilon = \hbar$ be a small parameter. Then the matrix-valued Wigner transform $W^{\varepsilon}$ of a mixed state density matrix $R^{\varepsilon}$ – whose eigenvalues (``occupation probabilities") $\lambda_j^{\varepsilon}$ satisfy $$\sum_{j=1}^{\infty} (\lambda_j^{\varepsilon})^2 \leq C(2\pi\varepsilon)^3,$$ and which solves the von Neumann equation associated with the Dirac equation \eqref{eq:dirac equation spinor} – converges weakly$^{\ast}$ in $L^{\infty}(\mathbb{R}_t,L^2(\mathbb{R}^3_x \times \mathbb{R}^3_{\xi}))$ to the Wigner matrix measure $W$, which is a solution of the matrix-valued Liouville equation \eqref{eq:matrix valued liouville introduction}-\eqref{eq:commutator constraint}. 

The evolution of the electron and positron densities $f_{\pm} := \tr_{\mathbb{C}^4}(W\Pi_{\pm})$, where $\Pi_{\pm}$ are the projections on the electron and positron subspaces of the symbol $P$ of the Dirac Hamilton operator $H$, is governed by the classical relativistic Hamilton functions \eqref{eq:classical hamiltonian}.
\end{theorem*}


\subsection{The relativistic Vlasov equation}
The relativistic Vlasov equation with given (i.e. external) time-dependent electromagnetic field for the phase space density $f_{\pm} (t,x,\xi) \colon \mathbb{R}_t \times \mathbb{R}^3_x \times \mathbb{R}^3_v \rightarrow \mathbb{R_+}$ including both electrons ($f_+$) and positrons ($f_-$) is given by 
\begin{align}
\label{eq:definition Vlasov Maxwell system original form}
\partial_t f_{\pm} + \frac{v}{\left< v \right>} \cdot \nabla_x 
f_{\pm}+ \left( E \pm \frac{v}{\left< v \right>} \times B \right) \cdot \nabla_v f_{\pm} &= 0 , \\
f_{\pm}|_{t=0}&=f_{\pm}^{\mathrm{in}}.
\end{align}
For the case of time dependent $A$, the electromagnetic field $E,B\in  \mathbb{R}^3$, is given by \begin{align}
    E =\nabla A_0 - \partial_t A, && B = \nabla \times A.
\end{align} 
with the magnetic vector potential $A\in \mathbb{R}^3$ and the electric scalar potential $A_0\in \mathbb{R}$ and the Lorentz force ($+$ for electrons, $-$ for positrons)
\begin{equation} \label{F_Lorentz}
    F_{\pm}^{\mathrm{L}} = E \pm \frac{v}{\left< v \right>} \times B .
\end{equation} 
Note that the \emph{relativistic velocity} $\frac{v}{\left< v\right>}$ is bounded by speed of light $c=1$ (in our semiclassical scaling) with the gamma factor $\frac{1}{\langle v \rangle}$ of special relativity, using the notation 
\begin{equation}
    \langle v \rangle = \sqrt{1 +|v|^2}.
\end{equation} 
Equation \eqref{eq:definition Vlasov Maxwell system original form} arises from the Liouville equation
\begin{equation}
    \partial_t f_{\pm}(t,x,\xi) = \{ H_{\pm}, f_{\pm} \}(t,x,\xi),
\end{equation}
with the relativistic Hamilton functions
\begin{equation}
    H_{\pm}(t,x,\xi) = \pm \langle\xi-A(t,x) \rangle - A_0(t,x).
\end{equation}
and the change of variables $v = \xi-A(t,x)$.

\subsection{The Dirac-von Neumann equation}
We use a scaling in \eqref{eq:dirac equation spinor} where $c=q=1$ and $\hbar = \varepsilon$. 
On the Hilbert space $\mathfrak{H} = L^2(\mathbb{R}^3, \mathbb{C}^4)$ we define the \emph{Dirac Hamilton operator}
\begin{equation}
    H = \alpha \cdot (-i\varepsilon \nabla -A(t,x)) + \beta -A_0(t,x)I_4.
    \label{eq: dirac hamiltonian}
\end{equation}
where $I_4$ denotes the identity in $\mathbb{C}^4$. 
We describe the mixed state of the relativistic quantum particle by its \emph{density operator}
$R=R^{\ast} \geq 0$  on $\mathfrak{H}$, cf. Section \ref{sec:notations}. 
We use the following version of the \emph{Dirac equation}, which is the \emph{von Neumann equation} given by the commutator of the Dirac Hamilton operator \eqref{eq: dirac hamiltonian} and the density operator:
\begin{align}
\label{eq:Dirac equation}
i \varepsilon \partial_t R
&= \left[ H   , R \right] , \quad R|_{t=0} = R_{\mathrm{in}}.
\end{align}
Note that for notational simplicity we omit the $\varepsilon$-dependence of $R$.
The Dirac Hamilton operator $H$ is the Weyl operator corresponding to the matrix-valued symbol
\begin{align}
    P(t,x,\xi) &= \alpha \cdot (\xi -A(t,x)) + \beta - A_0(t,x)I_4.
    \label{eq: time dependent electromagnetic Dirac symbol}
\end{align}
The four $4\times 4$ Dirac matrices\footnote{also called \emph{Gamma matrices} $\{\gamma^{\mu}\}_{\mu=0}^3$ in a different choice of basis for the Dirac algebra $\mathrm{Cl}_{1,3}(\mathbb{C})$, i.e. the complex Clifford algebra with signature $(1,3)$} $\beta, \alpha_k$, $k=1,2,3$ are given by
\begin{align}
\label{eq:dirac matrices}
\beta &= 
\begin{pmatrix}
I_{2} & 0 \\
0  & - I_{2}
\end{pmatrix}
\quad \text{and} \quad 
\alpha_k = 
\begin{pmatrix}
0 & \sigma_k  \\
\sigma_k & 0
\end{pmatrix}, \quad k=1,2,3
\end{align}
where $I_{2}$ denotes the identity in $\mathbb{C}^2$ and the $2 \times 2$ \emph{Pauli matrices} $\sigma_k$ representing the spin are given by 
\begin{align}
\sigma_1 &= 
\begin{pmatrix}
0 & 1 \\
1 & 0
\end{pmatrix} ,
\quad 
\sigma_2 = 
\begin{pmatrix}
0 & -i  \\
i & 0  
\end{pmatrix} 
\quad \text{and} \quad 
\sigma_3 = 
\begin{pmatrix}
1 & 0 \\
0 & -1 
\end{pmatrix} .
\end{align}
The dot product of a vector of matrices $\alpha = (\alpha_1,\alpha_2,\alpha_3)^\top$ with a vector $a$ is defined as
\begin{equation}
    \alpha\cdot a = \sum_{k=1}^3 \alpha_k a_k, \quad a\in \mathbb{R}^3.
\end{equation}
Note that $\alpha, \beta$ satisfy
\begin{align}
    \alpha_j\alpha_k + \alpha_k\alpha_j = 2\delta_{jk}I_4, && \alpha_j \beta + \beta \alpha_j = 0, && \beta^2 = I_4.
    \label{eq:clifford algebra}
\end{align}
We also introduce the auxilliary matrix
\begin{equation}
    \gamma^5 := -i \alpha_1\alpha_2\alpha_3.
\end{equation}
We collect some useful identities for the Dirac matrices in Section \ref{sec: identities gamma}. 

Let $R(t,x,y)$
be the integral kernel associated to the density operator $R$ in \eqref{eq:Dirac equation}.
The \emph{particle density} $\rho^{\varepsilon}$ and \emph{current density} $J^{\varepsilon}$ associated with $R$ are given by 
\begin{equation}
\label{eq:densities in terms of kernel}
    \rho^{\varepsilon}(t,x) = \tr_{\mathbb{C}^4}( R(t,x,x))
    \quad J^{\varepsilon}(t,x) = \tr_{\mathbb{C}^4}( \alpha R(t,x,x)).
\end{equation}
These expressions are well-defined elements of $L^1(\mathbb{R}^3_x)$ since $R$ is trace-class, cf. Section \ref{sec:notations}.
According to Section \ref{sec:notations}, the integral kernel $R(t,x,y)$ of the operator $R$ on $L^2$ has a spectral decomposition with eigenvalues $\lambda_j^{\varepsilon}$  that satisfy
\begin{equation}
    \lambda_j^{\varepsilon} \geq 0, \quad \sum_{j=1}^{\infty} \lambda_j^{\varepsilon} = 1,
\end{equation}
and a Hilbert basis $\{\psi_j^{\varepsilon}\}_{j\in \mathbb{N}}$ of $\mathfrak{H}$ such that
\begin{equation}
    R(t,x,y) = \sum_{j=1}^{\infty} \lambda_j^{\varepsilon} \psi_j^{\varepsilon}(t,x) \overline{\psi_j^{\varepsilon}(t,y)}^\top.
    \label{eq:R in eigenbasis}
\end{equation} 

The eigenvalues $\lambda_j^{\varepsilon}$ are the ``occupation probabilities" for each of the infinitely many possible states $\psi_j^{\varepsilon}$ of the "mixed state" represented by $R$. 

Note that for semiclassical limits of Wigner transforms in $L^2$ it is crucial to consider occupation probabilities $\lambda_j^{\varepsilon}$ that depend on the scaled Planck constant  $\varepsilon$, as already noted for the limit from the Schrödinger-Poisson to the Vlasov-Poisson equation \cite{lions1993mesures, markowich1993classical}.

Using \eqref{eq:R in eigenbasis} the {particle density} $\rho^{\varepsilon}$ and the {current density} $J^{\varepsilon}$ can be expressed as
\begin{equation}
    \rho^{\varepsilon}(t,x) = \sum_{j=1}^{\infty} \lambda_j^{\varepsilon} \langle \psi_j^{\varepsilon}(t,x), \psi_j^{\varepsilon}(t,x) \rangle_{\mathbb{C}^4}, \quad J^{\varepsilon}(t,x) = \sum_{j=1}^{\infty} \lambda_j^{\varepsilon} \langle \psi_j^{\varepsilon}(t,x), \alpha_k \psi_j^{\varepsilon}(t,x) \rangle_{\mathbb{C}^4} .\label{eq:charge current densities}
\end{equation}
For completeness we also give the definition of energy density $e^{\varepsilon}(t,x)$
\begin{equation}
\label{eq:energy density}
    e^{\varepsilon}(t,x) =\sum_{j=1}^{\infty} \lambda_j^{\varepsilon} \langle \psi_j^{\varepsilon}(t,x), H \psi_j^{\varepsilon}(t,x)\rangle_{\mathbb{C}^4},
\end{equation}
where $H$ is the Dirac Hamilton operator \eqref{eq: dirac hamiltonian}.

Denote
\begin{equation}
    P_0(t,x,\xi) = \alpha \cdot (\xi-A(t,x)) + \beta.
\end{equation}
Owing to \eqref{eq:clifford algebra} we observe that 
\begin{equation}
    P_0(t,x,\xi)^2 = (1+|\xi-A(t,x)|^2)I_4=\langle \xi-A(t,x)\rangle^2 I_4.
\end{equation}
Thus the minimal polynomial $m_P[\lambda]$ of the symbol $P(t,x,\xi) = P_0(t,x,\xi) -A_0(t,x)I_4$ is given by
\begin{equation*}
    m_P[\lambda](t,x,\xi) = (\lambda - \langle \xi-A(t,x)\rangle+A_0(t,x) )(\lambda+ \langle \xi-A(t,x)\rangle+A_0(t,x)).
\end{equation*}
Then the symbol $P$ has the eigenvalues
\begin{equation}
    \lambda_{\pm}(t,x,\xi) =  \pm \langle \xi-A(t,x)\rangle -A_0(t,x).
    \label{eq:eigenvalues time-dependent electromagnetic Dirac}
\end{equation}
Note that the eigenvalues $\lambda_{\pm}$ correspond to the classical relativistic Hamilton functions \eqref{eq:classical hamiltonian}.

Define
\begin{equation}
    S(t,x,\xi):=\frac{P_0(t,x,\xi)}{\langle\xi-A(t,x)\rangle} =S(t,x,\xi)^*, \quad S(t,x,\xi)^2 =  I_4.
\end{equation}
Then $S$ has eigenvalues $\pm 1$ with projections $\Pi_{\pm}$
\begin{align}
\label{eq:definition projections}
    \Pi_{\pm}(t,x,\xi) &= \frac{1}{2}\left(I_{4} \pm S(t,x,\xi)\right) = \frac{1}{2}\left(I_{4} \pm \frac{P_0(t,x,\xi)}{\langle\xi-A(t,x)\rangle}\right).
\end{align}
Therefore $P$ is given by
\begin{equation}
    P(t,x,\xi) =\lambda_+(t,x,\xi)\Pi_+(t,x,\xi) +\lambda_-(t,x,\xi)\Pi_-(t,x,\xi).
\end{equation}
In this paper, we consider the case where $A(t,x)$ and $A_0(t,x)$ are external, given potentials.

\subsection{The Dirac-Wigner equation}
Given $f,g \in {L^2}(\mathbb{R}^3,\mathbb{C}^4)$, we define the \emph{matrix-valued Wigner transform} or \emph{Wigner matrix} \cite{gerard1997homogenization} as
\begin{equation}
\label{eq:wigner transform general}
    W^{\varepsilon}[f,g](x,\xi) := \frac{1}{(2\pi)^3} \int_{\mathbb{R}^3} f(x + \frac{\varepsilon y}{2})\overline{g(x - \frac{\varepsilon y}{2})}^\top e^{-i\xi \cdot y} dy \in \mathcal{S}'(\mathbb{R}^3_x \times \mathbb{R}^3_{\xi},\mathbb{M}_4(\mathbb{C})).
\end{equation}
Given a density operator $R$ on $\mathfrak{H}=L^2(\mathbb{R}^3, \mathbb{C}^4)$ with integral kernel
\begin{equation}
    R(t,x,y) \in L^2(\mathbb{R}^3_x \times \mathbb{R}^3_y,\mathbb{M}_4(\mathbb{C})),
\end{equation}
we define the Wigner transform $W^{\varepsilon}[R]$ of $R$ as
\begin{align}
\label{eq:definition Wigner transform with spin}
W^{\varepsilon}[R](t,x,\xi) 
&=   \frac{1}{(2 \pi )^3}
\int_{\mathbb{R}^3}
R (t, x + \frac{\varepsilon y}{2} , x - \frac{\varepsilon y}{2} ) e^{- i \xi \cdot y}  dy  .
\end{align}
The charge and current densities \eqref{eq:charge current densities} can be expressed as the zeroth order moments in $\xi$ of the Wigner matrix $W^{\varepsilon}$
\begin{align}
    \rho^{\varepsilon}(t,x) = \int_{\mathbb{R}^3_{\xi}} \tr_{\mathbb{C}^4}(W^{\varepsilon}(t,x,\xi))d \xi, &&
    J^{\varepsilon}_k(t,x) = \int_{\mathbb{R}^3_{\xi}} \tr_{\mathbb{C}^4}(\alpha_k W^{\varepsilon}(t,x,\xi))d \xi.
\end{align}
The Wigner matrix $W^{\varepsilon}[R]$ then solves the \emph{Dirac-Wigner equation}
\begin{align}
\label{eq:Wigner matrix equation_2}
\partial_t W^{\varepsilon} &= \frac{1}{i\varepsilon}[P,W^{\varepsilon}]  +\frac{1}{2} \left( \left\{ P, W^{\varepsilon}\right\}-\left\{ W^{\varepsilon},P\right\} \right)  +  r^{\varepsilon} , \quad W^{\varepsilon}|_{t=0} = W^{\varepsilon}_{\mathrm{in}} = W^{\varepsilon}[R_{\mathrm{in}}],\\
P &= \alpha \cdot (\xi-A) + \beta-A_0,
\end{align}
where 
the \emph{remainder} $r^{\varepsilon}$ is given by
\begin{align}
\label{eq:remainder of Wigner matrix equation}
r^{\varepsilon} 
&= \varepsilon \Delta[A_k][\alpha_k,W^{\varepsilon}] +\frac{1}{2}(-\nabla_x A_k\cdot \nabla_{\xi}-\theta[A_k])[\alpha_k, W^{\varepsilon}]_+ 
 -(-\nabla A_0 \cdot \nabla_{\xi} -\theta [A_0])W^{\varepsilon} ,
\end{align}
where $[\,\cdot\,,\cdot\,]_+$ denotes the anticommutator, cf. Section \ref{sec:notations}. Here, $\theta$ is the pseudo-differential operator defined by
\begin{equation}
    (\theta[g]f)(x,\xi) := \int_{\mathbb{R}^6} \frac{1}{i\varepsilon}(g(x+\frac{\varepsilon y}{2})-g(x-\frac{\varepsilon y}{2}))f(x,\eta) e^{-i (\xi-\eta)\cdot y} d \eta d y, 
    \label{eq:PDO}
\end{equation}
while $\Delta$ is the pseudo-differential operator defined by
\begin{equation}
\label{eq:PDO_2}
    (\Delta[g]f)(x,\xi) := \frac{1}{2i}\int_{\mathbb{R}^6} \frac{g(x+\frac{\varepsilon y}{2})-2g(x)+g(x-\frac{\varepsilon y}{2})}{\varepsilon^2}f(x,\eta) e^{-i(\xi-\eta)\cdot y} d\eta d y.
\end{equation}
For the derivation of the Dirac-Wigner equation \eqref{eq:Wigner matrix equation_2} from the von Neumann equation \eqref{eq:Dirac equation} associated to the Dirac Hamilton operator \eqref{eq: dirac hamiltonian} we refer to Section \ref{subsection:Derivation of the Dirac-Wigner equation}.

\subsection{Analogy with Navier-Stokes}
\label{section analogy navier stokes}
The matrix-valued distribution $Y$ in \eqref{eq:matrix valued liouville introduction} should be thought of as the Lagrange multiplier associated with the constraint $[P,W]=0$. We observe an analogy with the
pressure $p\equiv p(t,x)\in\mathbb R$ in the Navier-Stokes equation for incompressible fluids
\begin{align*}
\partial_tu+u\cdot\nabla_xu+\nabla_xp&=\Delta_xu,\qquad x\in\mathbb R^3\\\nabla_x\cdot u &=0,
\end{align*}
where $u\equiv u(t,x)\in\mathbb R^3$ is the velocity field in the fluid. The pressure $p$ can be viewed as
the Lagrange multiplier needed to maintain the constraint $\nabla_x\cdot u=0$. In other words,
\begin{equation*}
Y\text{ is analogous to }p,\text{ and }-i[P,\cdot]\text{ is analogous to }\nabla_x.
\end{equation*}
Moreover, the pressure $p$ can be expressed solely in terms of $u$ and its derivatives. Indeed, taking the divergence of both sides of the motion equation,
\begin{equation}
-\Delta_xp=\nabla_x^2:(u\otimes u).
\label{eq:expression p in terms of u}
\end{equation}
There is an analogous formula for $Y$, 
  \begin{equation}
        Y = \frac{i}{4\langle \xi-A \rangle^2} ([\partial_t P,W] - \frac{1}{2}[P,\{P,W\}-\{W,P\}]),
        \label{eq:Y formula analogy}
    \end{equation}
    cf. equation \eqref{eq:Y formula} in Theorem \ref{thm:main1}.
    In other words, $Y$ is expressed in terms of $W$ analogously to how $p$  is expressed in terms of $u$. The motion equation for Navier-Stokes can be recast as
\begin{equation*}
\partial_tu+u\cdot\nabla_xu+\nabla_x(-\Delta)^{-1}\nabla_x\cdot(u\otimes u)=\Delta_xu,
\end{equation*}
whereas the matrix Liouville equation can be recast as
\begin{equation*}
    \partial_t W = \frac{1}{4\langle \xi-A \rangle^2} \big[P,[\partial_t P,W] - \frac{1}{2}[P,\{P,W\}-\{W,P\}]\big] + \frac{1}{2}\left(\{P,W\}-\{W,P\}\right)
\end{equation*}
Now if $u$ is a Leray solution of the Navier-Stokes equation in $\mathbb R^3$, one has
\begin{equation*}
u\in L^\infty(\mathbb R_+,L^2(\mathbb R^3,\mathbb R^3)),\quad\nabla_xu\in L^2(\mathbb R_+\times\mathbb R^3,\mathbb{M}_3(\mathbb C)),
\end{equation*}
so that, by Sobolev's embedding, $
u\in L^2(\mathbb R_+,L^6(\mathbb R^3))$. Therefore 
\begin{equation*}
u\otimes u\in L^2(\mathbb R_+,L^{3/2}(\mathbb R^3,\mathbb{M}_3(\mathbb C))),
\end{equation*}
so that by \eqref{eq:expression p in terms of u},
\begin{equation*}
    p\in L^2(\mathbb R_+;L^{3/2}(\mathbb R^3,\mathbb C)).
\end{equation*}
Hence
\begin{equation*}
\partial_tu\in L^2(\mathbb R_+;W^{-1,3/2}(\mathbb R^3,\mathbb R^3))+L^2(\mathbb R_+;H^{-1}(\mathbb R^3,\mathbb R^3)),
\end{equation*}
so that, for each $\chi\equiv \chi(x)\in C^\infty_c(\mathbb R^3)$,
\begin{equation*}
(t,x)\mapsto\chi(x)u(t,x)\in C(\mathbb R_+;W^{-1,3/2}(\mathbb R^3,\mathbb R^3)).
\end{equation*}
In particular, the trace $u_0 := u(0,\cdot)$ is a well-defined element of $W^{-1,3/2}_{{\mathrm{loc}}}(\mathbb R^3,\mathbb R^3)$. The fact that the initial data $u_0$ can be defined by utilizing \eqref{eq:expression p in terms of u} motivates a similar reasoning for the Wigner measure $W$. So far, the convergence of the Wigner transform $W^{\varepsilon}$ to the Wigner measure $W$ is stated in the sense of distributions in $t,x,\xi$ (cf. Proposition \ref{proposition general}) and the presence of the Lagrange multiplier $Y$ in the matrix-valued Liouville equation inhibits the prescription of initial data at $t=0$. But formula \eqref{eq:Y formula analogy} gives suitable estimates for $Y$ in terms of $W$ (Lemma \ref{thm:characterization Y}) such that the prescription of initial data $W_{\mathrm{in}}$ for \eqref{eq:matrix valued liouville introduction} becomes possible (Lemma \ref{thm:initial data}). The reader interested in a more detailed discussion of the Navier-Stokes equation used in this analogy is referred to section 1.8 in chapter 1 of \cite{MajdaBerto}.

\section{Main results}
We have the following main results. The first theorem establishes the semi-classical limit of a solution to the Dirac-Wigner equation \eqref{eq:Wigner matrix equation_2}-\eqref{eq:remainder of Wigner matrix equation}, i.e. the Wigner transform of the Dirac-von Neumann equation  \eqref{eq:Dirac equation}, to the Wigner matrix measure $W$, satisfying the matrix-valued Liouville equation \eqref{eq:matrix valued liouville introduction}.
\begin{theorem}\label{thm:main1}
Let $A\in C^1(\mathbb{R}_t,H^s(\mathbb{R}^3_x))$, where $s >5/2$ is a real number, and let $A_0 \in C^1(\mathbb{R}_t,H^1(\mathbb{R}^3))$.  Let  $ R_{\mathrm{in}} \in \textfrak{S}^{1}(\mathfrak{H})$ be a density operator such that
\begin{equation}
\label{eq:assumption lambda}
\norm{R_{\mathrm{in}}}_{\textfrak{S}^{2}(\mathfrak{H})}^2 = \sum_{j=1}^{\infty} (\lambda_j^{\varepsilon})^2 \leq C(2\pi\varepsilon)^{{3}},
\end{equation}
\begin{equation}
\tr_{\mathfrak{H}}\left( \sqrt{1 - \Delta} R_{\mathrm{in}} \sqrt{1 - \Delta} \right) < + \infty.
\end{equation}
Let $R$ be the unique solution to the Dirac-von Neumann equation \eqref{eq:Dirac equation} with initial datum $R_{\mathrm{in}}$ and let $W^{\varepsilon}= W^{\varepsilon}[R]$ be its Wigner transform. Assume that 
\begin{equation}
W_{\pm}^{\mathrm{in},\varepsilon} := \Pi_{\pm} W^{\mathrm{in},\varepsilon} \Pi_{\pm} \underset{\varepsilon \rightarrow 0}{\longrightarrow} \Pi_{\pm} W^{\mathrm{in}} \Pi_{\pm} =: W^{\mathrm{in}}_{\pm} \quad \text{ weakly in } L^2(\mathbb{R}^3_x \times \mathbb{R}^3_{\xi},\mathbb{M}_4(\mathbb{C})),
\label{eq:convergence initial data}
\end{equation}
where $\Pi_{\pm}$, defined in \eqref{eq:definition projections}, are the projections onto the electron and positron subspaces of $P$.
Then for any $T>0$, $W^{\varepsilon}$ converges weakly$^{\ast}$ in $L^{\infty}(\mathbb{R}_t, L^2(\mathbb{R}^6_{x,\xi},\mathbb{M}_4(\mathbb{C})))$ to a non-negative matrix-valued Radon measure $W \in L^{\infty}(\mathbb{R}_t,L^2(\mathbb{R}^6_{x,\xi},\mathbb{M}_4(\mathbb{C})))$, called the Wigner matrix measure.  Moreover, there exists 
\begin{equation}
    Y \in H^{-1}_{\mathrm{loc}}(\mathbb{R}_t\times\mathbb{R}^3_x\times\mathbb{R}^3_{\xi},\mathbb{M}_4(\mathbb C)),
\end{equation}
such that $W$ solves the matrix-valued Liouville  equation 
\begin{align}
\label{eq:main theorem transport equation}
        \partial_t W &= -i[P,Y] + \frac{1}{2}(\{P,W\}-\{W,P\}), \quad P = \alpha \cdot (\xi-A) + \beta -A_0,
        \\
        [P,W] &= 0, \\
        W|_{t=0} &= W^{\mathrm{in}}_+ + W^{\mathrm{in}}_-,
    \end{align}
    in $H^{-1}_{\mathrm{loc}}(\mathbb{R}_t \times \mathbb{R}^3_x \times \mathbb{R}^3_{\xi})$. The matrix-valued distribution $Y$ satisfies $Y=Y^*$, $\Pi_{\pm} Y \Pi_{\pm} = 0$ and is given by
    \begin{equation}
        Y = \frac{i}{4\langle \xi-A \rangle^2} \left([\partial_t P,W] - \frac{1}{2}[P,\{P,W\}-\{W,P\}]\right).
        \label{eq:Y formula}
    \end{equation}
    Finally, we have that
    \begin{equation}
        \iint \tr(W(t,x,\xi)) dx d\xi = 1 \quad \text{ for almost every } t,
    \end{equation}
    and
    \begin{align}
        \rho^{\varepsilon}(t,\cdot) = \int \tr(W^{\varepsilon}(t, \cdot , \xi)) d \xi &\underset{\varepsilon \rightarrow 0}{\longrightarrow}  \int \tr(W(t, \cdot , \xi)) d \xi =: \rho(t,\cdot), \\
        J_k^{\varepsilon}(t,\cdot) = \int \tr(\alpha_k W^{\varepsilon}(t, \cdot , \xi)) d \xi &\underset{\varepsilon \rightarrow 0}{\longrightarrow}  \int \tr(\alpha_k W(t, \cdot , \xi)) d =: J_k(t,\cdot),
    \end{align}
    narrowly on $\mathbb{R}^3_x$ for almost every $t$.
    \end{theorem}
    Note that here we take the limit in the ``full" Wigner equation for $W^{\varepsilon}$ instead of the projected Wigner transform $\Pi_{\pm} W^{\varepsilon} \Pi_{\pm}$ as in \cite{gerard1997homogenization, spohn2000semiclassical}. In fact it seems natural to refer to the projections $\Pi_{\pm}$ as little as possible since they are nonlinear functions of $x,\xi$ and thus their Weyl quantization is very complicated to calculate.
    
    We retain the term $-i[P,Y]$ (i.e. the Lagrange multiplier associated with $[P,W]=0$) as the limit of the oscillating term $\tfrac{1}{i\varepsilon}[P,W^{\varepsilon}]$, which arises as the $O(\tfrac{1}{\varepsilon})$ term in the expansion of the Moyal bracket. It is non-zero in the matrix-valued case since the matrix-valued symbol $P$ and the matrix-valued Wigner transform $W^{\varepsilon}$ do not commute in general. 
Due to Theorem \ref{thm:main1}, applying $\Pi_{\pm}$ to the matrix-valued Liouville equation \eqref{eq:main theorem transport equation} from both sides yields
\begin{align}
     \Pi_{\pm} \partial_t W \Pi_{\pm} &= \frac{1}{2} \Pi_{\pm}\left(\left\{ P,W\right\} -\left\{W, P\right\} \right)\Pi_{\pm}, \label{eq:Wigner matrix equation limit projected} \\
     [\Pi_{\pm},W] &=0.
\end{align}
It remains to extract the equations for the electron and positron Wigner matrix measures $W_{\pm}= \Pi_{\pm} W{\Pi_{\pm}}$ and the scalar electron and positron Wigner measures $f_{\pm} = \tr_{\mathbb{C}^4}(W_{\pm})$. This happens at the cost of additional terms, which are referred to as \emph{Berry phase} $H_{\mathrm{be}}$ and \emph{Poissonian curvature} $H_{\mathrm{pc}}$ (or \emph{``no-name" term}, coined in \cite{spohn2000semiclassical}), cf. \cite{emmrich1996geometry, littlejohn1991geometric, spohn2000semiclassical}. This is the content of our second main theorem.
\begin{theorem}
    \label{thm:main2}  
Let $W_{\pm}^{\mathrm{in}} = \Pi_{\pm} W_{\mathrm{in}} \Pi_{\pm}$, where $\Pi_{\pm}$ is given by \eqref{eq:definition projections}. Under the assumptions of Theorem \ref{thm:main1} it follows that $W_{\pm} := \Pi_{\pm} W \Pi_{\pm}$ solves
    \begin{align}
    \partial_t W_{\pm}
    &= \{\lambda_{\pm},W_{\pm}\} + [H_{\mathrm{be}} + H_{\mathrm{pc}},W_\pm], \quad W_{\pm}|_{t=0} = W_{\pm}^{\mathrm{in}} \quad \text{in } H^{-1}_{\mathrm{loc}},\label{eq:relativistic Vlasov}
\end{align}
where $\lambda_{\pm}$ is given by \eqref{eq:eigenvalues time-dependent electromagnetic Dirac} and where the \emph{Berry term} $H_{\mathrm{be}}$ and the \emph{Poissonian curvature} $H_{\mathrm{pc}}$ are given by
\begin{equation}
\label{eq:berry term}
    H_{\mathrm{be}} := [\Pi_{\pm},\{\lambda_{\pm}
    ,\Pi_\pm\}-\partial_t\Pi_{\pm}],
\end{equation}
\begin{equation}
\label{eq:poissonian curvature}
    H^{\pm}_{\mathrm{pc}} :=   \pm\langle \xi-A\rangle \Pi_{\pm}\{\Pi_\mp,\Pi_\mp\}\Pi_{\pm} ,
\end{equation}
 In particular, the scalar Wigner measures $f_{\pm}(t,x,v) := \tr(W_{\pm})(t,x,\xi-A(t,x))$ solve the relativistic Vlasov equations with Lorentz force and external magnetic and electric fields $A,A_0$ for the electron and positron phase space densities $f_+$ and $f_-$, respectively, i.e. $f_{\pm}$ solves
    \begin{align}
    &\partial_t f_{\pm} \pm \frac{v}{\langle v \rangle}\cdot \nabla_x f_{\pm} + (E\pm \frac{v}{\langle v \rangle} \times B) \cdot \nabla_v f_{\pm} = 0, \quad f_{\pm}|_{t=0} = f_{\pm}^{\mathrm{in}} := \tr(W_{\pm}^{\mathrm{in}}),\label{eq:relativistic transport equation} \\
    &E= \nabla{A_0} - \partial_t A, \quad  B=\nabla \times A.
    \nonumber
\end{align}
in $H^{-1}_{\mathrm{loc}}$ with the charge and current densities
\begin{align}
    \rho(t,x) = \int_{\mathbb{R}^3} (f_++f_-)(t,x,v) d v, && J_k(t,x) = \int_{\mathbb{R}^3} \frac{v}{\langle v \rangle}(f_+ -f_-)(t,x,v) d v.
\end{align}
\end{theorem}

For completeness we give explicit expressions for the Berry term $H_{\mathrm{be}}$ and the Poissonian curvature $H_{\mathrm{pc}}$. Recall that $v=\xi-A(t,x)$. The Berry term $H_{\mathrm{be}}$ is given by
\begin{align}
    H_{\text{be}} &= -\frac{1}{2\langle v \rangle^2}(\beta\alpha\cdot F^{\mathrm{L}}_{\pm} -i\gamma^5 \alpha \cdot (v\times F^{\mathrm{L}}_{\pm})), 
 \\ F^{\mathrm{L}}_{\pm} &= E \pm \frac{v}{\langle v\rangle}\times B = \nabla A_0 -\partial_t A \pm \frac{v}{\langle v\rangle}\times (\nabla \times A),
\end{align}
where $F^{\mathrm{L}}_{\pm}$ denotes the Lorentz force for electrons and positrons, cf. \eqref{F_Lorentz}.
Note that the electric field in the definition of $F^{\mathrm{L}}_{\pm}$ contains the magnetic contribution $\partial_t A$ which comes from the term $\partial_t \Pi_j$ and which was omitted in \cite{gerard1997homogenization, spohn2000semiclassical} since their symbol $P$ does not depend on $t$. 

The Poissonian curvature $H_{\mathrm{pc}}$ is given by
\begin{equation}
    H_{\text{pc}} =  \pm\frac{i}{2\langle v \rangle} \Pi_{\pm} \left( \gamma^5 (\alpha \cdot B )\right) \Pi_{\pm} .
\end{equation}
The explicit derivations of these expressions can be found in the Appendix in Sections \ref{sec:berry} and \ref{sec:poissonian curvature}.


\begin{remark}
    As mentioned in the introduction, for $C^{\infty}$, time-independent potentials the limit was established in \cite{gerard1997homogenization,spohn2000semiclassical}. Here, we calculate the remainder $r^{\varepsilon}$ explicitly in order to allow for rougher, time-dependent potentials and use a Lagrange multiplier approach in order to avoid projecting the Wigner equation onto the electron and positron subspaces before passing to the limit.
\end{remark}
\begin{remark}
One could reduce the regularity needed for the Wigner transform (in particular allowing for pure states) by increasing the regularity for the electromagnetic potentials $A_0,A$, but this is not the main goal of this paper. On the other hand let us comment on the regularity chosen in Theorem \ref{thm:main1}: in the classical limit considered in this work, the trace of the projected (matrix-valued) Wigner measure associated to the family of density operators satisfying the Dirac-von Neumann equation \eqref{eq:Dirac equation} is a solution of the relativistic Vlasov equation \eqref{eq:relativistic transport equation} with external electromagnetic field defined in terms of $A_0,A$.
In order for the Cauchy problem associated with \eqref{eq:relativistic transport equation} to be well-posed – in particular, in order for a classical solution of the Cauchy problem to exist and be unique for all smooth initial
data with sufficient decay at infinity – it is natural to impose on the electromagnetic field the conditions which guarantee that one can apply the method of characteristics.

Characteristic curves of the Liouville equation above are integral curves of the differential system
\[
\dot x(t)=\frac{v(t)}{\langle v(t) \rangle}, \quad\dot v(t)=E(t,x(t))\pm \frac{v(x(t))}{\langle v(x(t))\rangle}\times B(t,x(t))\,.
\]
In order to apply the Cauchy-Lipschitz theorem to conclude that the Cauchy problem for this differential system has a unique global solution defined on the time interval $[0,T]$ for any
initial position and momentum, it is natural to require that
\[
\partial_tA,\,\,\,\nabla A,\,\,\,\nabla_xA_0,\,\,\,\nabla \partial_tA,\,\,\,\nabla^2A\text{ and }\nabla^2A_0\text{ are continuous on }[0,T]\times\mathbb R^3.
\]
See for instance \cite[Theorem 1.1 in Chapter 1]{bouchut2000kinetic}, or  \cite[Theorem 1.2.2]{hormander1997lectures}, or even  \cite[Chapter II.1]{cartan1971differential}.

The assumptions in the present paper are therefore natural in light of these requirements pertaining to the existence of the characteristic flow for the limiting Liouville equation describing
the dynamics in the classical limit.
\end{remark}
\begin{remark}
\label{remark L^2}
Note that Assumption \eqref{eq:assumption lambda} does not hold for pure states on $\mathfrak{H}$.  To see this recall that the kernel of $R$ on $\mathfrak{H}$ can be represented by 
\begin{equation}
 R (t,x,y)= \sum_{j=1}^{\infty} \lambda_j \psi_j(t,x) \overline{\psi_j(t,y)}^\top, \quad \lambda_j\geq 0, \quad \sum_{j=1}^{\infty} \lambda_j = 1,
\end{equation}
where $\{\psi_j\}_{j\in \mathbb{N}}$ is orthonormal in $\mathfrak{H}$, cf. Section \ref{sec:notations}. Using the definition of the Wigner transform \eqref{eq:definition Wigner transform with spin} and Plancherel's theorem, it is easy to see that 
\begin{equation*}
\norm{R_{\mathrm{in}}}_{\textfrak{S}^{2}}^2 = \sum_{j=1}^{\infty}\lambda_{j}^2\le (2\pi\varepsilon)^3\sup_{0<\varepsilon\le 1}\|W^{\varepsilon}\|_{L^2}^2.
\end{equation*}
By the Cauchy-Schwarz inequality
\begin{equation*}
    1=\left(\sum_{j=1}^{\infty}\lambda_{j}\right)^2\le\sum_{j=1}^{\infty}\mathbf 1_{\lambda_{j}>0}\sum_{j=1}^{\infty}\lambda_{j}^2\le\text{rank}R\cdot(2\pi\varepsilon)^3\sup_{0<\varepsilon\le 1}\|W^\varepsilon\|_{L^2}^2,
\end{equation*}
Therefore, assuming that a family $\{R^{\varepsilon},:\,0<\varepsilon\le 1\}$ of density operators on $L^2(\mathbb R^d)$ satisfies
\begin{equation*}
\sup_{0<\varepsilon\le 1}\|W^\varepsilon\|_{L^2}^2=C<\infty,
\end{equation*}
implies that 
\begin{equation*}
    \text{rank}R^{\varepsilon}\ge\frac1{(2\pi\varepsilon)^3C},\qquad\text{ so that }\varliminf_{\varepsilon\to 0}\left(\varepsilon^3\text{rank}R^\varepsilon\right)>0.
\end{equation*}
In particular, this assumption rules out the possibility that $\{R^\varepsilon\,:\,0<\varepsilon\le 1\}$ is a family of pure states (i.e. rank-one density operators), cf. Section \ref{sec:notations}.

\end{remark}

\begin{remark}
In the self-consistent case of Dirac-Maxwell, $A_k^{\varepsilon}$, $A_0^{\varepsilon}$ (now also depending on ${\varepsilon}$) are given by Maxwell's equations (here in Coulomb gauge $\nabla \cdot A=0$), 
\begin{align}
    \Delta A_0^{\varepsilon}(t,x) &= \rho^{\varepsilon}(t,x) = \int \tr(W^{\varepsilon}(t,x,\xi)) d\xi, 
\\
\left( \partial_t^2 - \Delta \right) A_k^{\varepsilon}(t,x) 
&= \mathbb{P} J^{\varepsilon}_{k}(t,x) =\int \tr(\alpha_k W^{\varepsilon}(t,x,\xi) )d \xi,
\end{align}
where $\mathbb{P}$ denotes the Leray projection onto divergence-free vector fields.
One needs to extract some sort of strong convergence of $A_k^{\varepsilon},A_0^{\varepsilon}$ as $\varepsilon \rightarrow 0$ in order to be able to pass to the limit (due to the weak convergence of $W^{\varepsilon}$). On the other hand, $\rho^{\varepsilon}$ and $J^{\varepsilon}$ are only $L^1$ uniformly in $\varepsilon$. Since the densities can be expressed as moments of the Wigner transform $W^{\varepsilon}$ (or traces thereof), additional integrability for the densities could maybe come from interpolation with the energy (such as for the Schrödinger-Poisson equation, cf. \cite{lions1993mesures, markowich1993classical}), but the energy \eqref{eq:energy density} of the Dirac equation does not have a definite sign, which is also a major obstacle to the global wellposedness of the Dirac-Maxwell equation. 
\end{remark}



\subsection{Notation}
\label{sec:notations}
The space of bounded measures on $\mathbb{R}^d$ is denoted by $\mathcal{M}(\mathbb{R}^d)$. The space of tempered distributions is denoted by $\mathcal{S}'$ where $\mathcal{S}$ is the Schwartz space.

Let 
\begin{equation}
    \mathfrak{H}^{d,m} = L^2(\mathbb{R}^d, \mathbb{C}^m).
\end{equation}
We denote the Hilbert space of Dirac 4-spinors on $\mathbb{R}^3$ by
\begin{equation}
    \mathfrak{H} \equiv \mathfrak{H}^{3,4} = L^2(\mathbb{R}^3, \mathbb{C}^4).
\end{equation}
Let $\mathbb{M}_m(\mathbb{C})$ be the space of $m\times m$ matrices with entries in $\mathbb{C}$,
\begin{equation}
    \langle v , w \rangle_{\mathbb{C}^m} = \sum_{j=1}^m \overline{v_j}w_j, \quad v,w \in \mathbb{C}^m
\end{equation}
be the inner product on $\mathbb{C}^m$, and $v^\top$ be the transpose of $v$.

We use the following notation for the so-called ``Japanese bracket"
\begin{equation}
    \langle x \rangle = \sqrt{1 +|x|^2}.
\end{equation} 

The set of bounded operators on $\mathfrak{H}$ is denoted by $\textfrak{S}^{\infty} (\mathfrak{H})$. The set of trace-class operators on $\mathfrak{H}$ is denoted by $\textfrak{S}^1( \mathfrak{H})$. More generally, for $p\in[1,\infty)$, we denote by $\textfrak{S}^p(\mathfrak{H})$ the $p$-Schatten space equipped with the norm $$\norm{A}_{\textfrak{S}^p}^p=\tr_{\mathfrak{H}}( \abs{A}^p), \quad \abs{A}=\sqrt{A^*A}.$$

In particular we denote the trace in $\mathfrak{H}$ by $\tr_{\mathfrak{H}}(\cdot)$ and the trace on $\mathbb{C}^4$ by $\tr_{\mathbb{C}^4}(\cdot)$, In particular, for a trace-class operator $A \in \textfrak{S}^1(\mathfrak{H})$ with kernel $K_A(x,y)\in L^2(\mathbb{R}^3_x\times\mathbb{R}^3_y,\mathbb{M}_4(\mathbb{C}))$ we have
\begin{equation}
    \tr_{\mathfrak{H}}(A) = \int_{\mathbb{R}^3_x} \tr_{\mathbb{C}^4}(K_A(x,x)) dx = \int_{\mathbb{R}^3_x} \sum_{j=1}^4 (K_A)_{jj}(x,x) dx.
\end{equation}

A \emph{density operator} $R\in \textfrak{S}^1( \mathfrak{H})$ is a self-adjoint ($R=R^*$), nonnegative ($\langle R\psi,\psi\rangle \geq 0$ for all $\psi \in \mathfrak{H}$) operator 
with trace $$\tr_{\mathfrak{H}}(R) = \int_{\mathbb{R}^3_x} \tr_{\mathbb{C}^4}(R(x,x)) dx = \int_{\mathbb{R}^3_x} \sum_{j=1}^4 R_{jj}(x,x) dx = 1.$$ A rank-one density operator is a \emph{pure state}.  On $\mathfrak{H}$, the operator $R$ is an integral operator 
with integral kernel $(x,y)\mapsto R(x,y) \in L^2(\mathbb{R}_x^d\times \mathbb{R}_y^d,\mathbb{M}_4(\mathbb{C}))$, which we abusively denote by the same letter. Since $R=R^*$ is trace-class, it is compact and by the spectral theorem 
there is a sequence of real eigenvalues $(\lambda_j)_{j\ge 1}$ and a Hilbert basis (i.e. a complete orthonormal system) $\{\psi_j\,:\,j\ge 1\} \subset \mathfrak{H}$ such that
\begin{equation}\label{SpecDecR}
    R(x,y) = \sum_{j=1}^{\infty} \lambda_j \psi_j(x) \overline{\psi_j(y)}^\top.
\end{equation}
Since $R$ is a density operator,
\begin{equation}\label{CondLambda}
\lambda_j \geq 0\quad\text{since }R=R^*\ge 0,\quad\text{ and }\sum_{j=1}^{\infty}\lambda_j 
=\tr_{\mathfrak{H}}(R)=1,
\end{equation}
since $\{\psi_j\}$ is an orthonormal family. In particular, the restriction of the integral kernel to the diagonal, 
i.e. $R(x,x)$, is a well-defined element of $L^1(\mathbb{R}^d_x, \mathbb{M}_4(\mathbb{C}))$ and we can define the density function $\rho\in L^1(\mathbb{R}_x^d)$ as in \eqref{eq:densities in terms of kernel}.

For two matrices $A,B$, the commutator is defined as $[A,B] = AB-BA$ and the anticommutator as $[A,B]_+ = AB+BA$.  For two (not necessarily commuting) matrix-valued functions $F,G$ on phase space $\mathbb{R}^d_x \times \mathbb{R}^d_{\xi}$ we define the Poisson bracket as
\begin{equation}
    \{F,G\} := \sum_{k=1}^d \partial_{x_k} F \partial_{\xi_k} G-\partial_{\xi_k} F \partial_{x_k} G,
\end{equation}
where the derivatives are taken in the sense of distributions. Note that the order is important and in general, $\{F,G\}\neq-\{G,F\}$. Depending on the context, $\{\,\cdot\,,\cdot\,\}$ may also denote the standard Poisson bracket for scalar functions.
On the space $\mathcal{S}'(\mathbb{R}_x^d\times \mathbb{R}^d_{\xi},\mathbb{M}_m(\mathbb{C}))$ we define the duality bracket
\begin{equation}
    \langle A,B\rangle := \int \tr(A^{\ast} B) dx d \xi.
\end{equation}

The letter $C$ denotes a general positive constant that may vary from one line to another. 
\section{Proof of the main theorems}

\subsection{Outline of the proof}

In Section \ref{sec:preliminaries} we prove the existence of solutions to the Dirac-von Neumann equation \eqref{eq:Dirac equation} and the conservation of mass for suitably regular time-dependent external fields. Moreover, we show that the remainer $r^{\varepsilon}$ in the expansion of the Moyal product on the right hand side of the Dirac-Wigner equation \eqref{eq:Wigner matrix equation_2} converges to zero. Finally, we gather some technical estimates on the regularity of the projections $\Pi_{\pm}$. 

In Section \ref{sec:proofmain1} we prove Theorem \ref{thm:main1}, which we split into multiple Lemmas. First we prove the convergence of Wigner equations with general symbols $P$ satisfying some regularity assumptions towards the matrix-valued Liouville equation with Lagrange multiplier $X$. Then we examine the structure of $X$ in the case of the Dirac equation and show that it is given by $X=-i[P,Y]$ with $Y$ given by \eqref{eq:Y formula}. The fact that $Y$ depends only on $P$ and $W$ allows for the prescription of initial data $W_{\mathrm{in}}$, which is identified as the limit of the data $W^{\varepsilon}_{\mathrm{in}}$ of the Dirac-Wigner equation. Finally, we prove the convergence of the charge and current densities and the tightness of the sequence $R^{\varepsilon}(x,x)$.

In Section \ref{sec:proofmain2} we prove Theorem \ref{thm:main2}, where we exploit the algebraic properties of the projections $\Pi_{\pm}$ in order to obtain the relativistic transport equations for the densities $f_{\pm}$.

\subsection{Preliminaries}
\label{sec:preliminaries}
We first discuss the question of existence of solutions to the Dirac-von Neumann equation \eqref{eq:Dirac equation} for external time-dependent fields. 

\begin{proposition}
\label{thm:existence dirac}
Let $r >3$ and $A_{\mu} = B_{\mu} + \tilde{B}_{\mu}$  with $B_{\mu} \in C^1(\mathbb{R}, L^r(\mathbb{R}^3,\mathbb{R}))$ and  $\tilde{B}_{\mu} \in C^1(\mathbb{R}, L^{\infty}(\mathbb{R}^3,\mathbb{R}))$ for $\mu = 0,1,2,3$.  Then for all $t \in \mathbb{R}$ the Dirac Hamilton operator \eqref{eq: dirac hamiltonian} with domain $\mathcal{D}(H(t)) = H^1(\mathbb{R}^3,\mathbb{C}^4)$ is self-adjoint. Moreover, there exists a  jointly strongly continuous two-parameter family  of unitary propagators $U_{H}(t;s)$ satisfying 
\begin{align}
i \partial_t U_H(t;s) \psi = H(t) U_H(t;s) \psi  
\quad \text{for all} \;  \psi \in H^1(\mathbb{R}^3,\mathbb{C}^4),
\end{align}
and $U_H(s;s) = 1$ for all $s \in \mathbb{R}$.
\end{proposition}

\begin{proof}
Let $\widetilde{H} = -i\varepsilon  \alpha \cdot \nabla  + \beta $ denote the free Dirac Hamiltonian which is a self-adjoint operator on the domain $H^1(\mathbb{R}^3,\mathbb{C}^4)$ \cite[Theorem 1.1]{thaller1992dirac}.
Moreover, let $f \in H^1(\mathbb{R}^3, \mathbb{C})$. By means of H\"older's inequality,  the interpolation inequality and Sobolev's inequality one can prove that for every $r >3$ there exists $\theta_r \in (0,1)$ such that for all $\mu \in \{0,1,2,3 \}$
\begin{align}
\label{eq:proof existence of unitary group auxilliary bound}
\norm{B_{\mu}(t,\cdot) f}_{L^2(\mathbb{R}^3,\mathbb{C})} &\leq C \norm{B_{\mu}(t,\cdot) }_{L^r(\mathbb{R}^3,\mathbb{C})}  \,   \norm{f} _{L^2(\mathbb{R}^3,\mathbb{C})}^{1- \theta_r}
\norm{f}_{H^1(\mathbb{R}^3, \mathbb{C})}^{\theta_r}  .
\end{align}  
Together with Young's ineqality for products this leads to ($\psi \in H^1(\mathbb{R}^3,\mathbb{C}^4)$) 
\begin{align}
&\big\|(H(t) - \widetilde{H}) \psi \big\|_{\mathfrak{H}}
\nonumber \\
&\quad \leq a_{\theta_r} \norm{\psi}_{H^1(\mathbb{R}^3,\mathbb{C}^4)} +  \sup_{\mu \in \{0,1,2,3,\}} \Big( b_{\theta_r} \norm{B_{\mu}(t,\cdot) }_{L^r(\mathbb{R}^3,\mathbb{C})}^{\frac{1}{1 - \theta_r}}
+ \norm{\tilde{B}_{\mu}(t,\cdot) }_{L^{\infty}(\mathbb{R}^3,\mathbb{C})} \Big)
 \norm{\psi}_{\mathfrak{H}}
\end{align}
with $0 < a_{\theta_r} < 1 $ and $b_{\theta_r} > 0$.  By the Kato--Rellich theorem this proves the self-adjointness of $H(t)$ for all $t \in \mathbb{R}$.  Note that \eqref{eq:proof existence of unitary group auxilliary bound} implies 
\begin{align}
&\norm{\Big( h^{-1} \big( H(t+h) - H(t) \big) + \alpha \cdot \dot{A}(t,\cdot) + \dot{A}_0(t, \cdot) I_{4} \Big) \psi}_{\mathfrak{H}}
\nonumber \\
&\quad \leq C \norm{\psi}_{H^1(\mathbb{R}^3,\mathbb{C}^4)} 
\sum_{\mu = 0}^3
\Big(
\norm{h^{-1} \big( B_{\mu}(t+h,\cdot) - B_{\mu}(t,\cdot) \big)
- \dot{B}_{\mu}(t, \cdot) }_{L^r(\mathbb{R}^3, \mathbb{C})} 
\nonumber \\
&\qquad \qquad \qquad \qquad \qquad + 
\norm{h^{-1} \big( \tilde{B}_{\mu}(t+h,\cdot) - \tilde{B}_{\mu}(t,\cdot) \big)
- \dot{\tilde{B}}_{\mu}(t, \cdot) }_{L^{\infty}(\mathbb{R}^3, \mathbb{C})} 
\Big) .
\end{align}
Using that both $B_{\mu}$ and $\tilde{B}_{\mu}$ are continuous differentiable in time let us conclude the strong differentiability of the mapping $t \mapsto H(t)$, proving the existence of the two-parameter family of propagators (see \cite[Theorem X.70]{ReedSimon1975} and \cite[Theorem 2.2]{GS2014}).
\end{proof}

\begin{corollary}
\label{thm:corollary existence}
Let $(A_0,A)$ and $U_{H}$ be as in Proposition~\ref{thm:existence dirac}.  Moreover,  let  $ R_{\mathrm{in}} \in \mathfrak{S}^{2}(\mathfrak{H})$ be a density operator such that $$\tr_{\mathfrak{H}}\left( \sqrt{1 - \Delta} R_{\mathrm{in}} \sqrt{1 - \Delta} \right) < + \infty.$$ Then  \begin{equation}
\label{eq:evolution R}R(t) := U_H(t;0) R_{\mathrm{in}} U_H(0;t),\end{equation} is the unique $C^1 ( \mathbb{R},  \mathfrak{S}^{2}(\mathfrak{H}) )$ solution to the Dirac-von Neumann equation~\eqref{eq:Dirac equation} with initial datum $R_{\mathrm{in}}$.
\end{corollary}
By Plancherel's theorem we deduce the following Corollary.
\begin{corollary}
\label{thm:existence wigner}
Let $(A_0,A)$ be as in Proposition \ref{thm:existence dirac} and $R_{\mathrm{in}}$ as in Corollary \ref{thm:corollary existence}. Then the Wigner transform $W^{\varepsilon}(t)= W^{\varepsilon}[R(t)]$, where $R(t)$ is given by \eqref{eq:evolution R}, satisfies the Dirac-Wigner equation \eqref{eq:Wigner matrix equation_2}-\eqref{eq:remainder of Wigner matrix equation} in the sense of distributions, where $W^{\varepsilon}(0) = W^{\varepsilon}_{\mathrm{in}} = W^{\varepsilon}[R_{\mathrm{in}}]$. The evolution of $W^{\varepsilon}(t)$ conserves mass, i.e.
\begin{equation}
    \|W^{\varepsilon}(t,\cdot,\cdot)\|_{L^2(\mathbb{R}_x^3\times \mathbb{R}^3_{\xi})} = \|W^{\varepsilon}_{\mathrm{in}}\|_{L^2(\mathbb{R}_x^3\times \mathbb{R}^3_{\xi})}.
\end{equation}
\end{corollary}

Next we show that the remainder $r^{\varepsilon}$ in the Dirac-Wigner equation \eqref{eq:Wigner matrix equation_2} converges to zero for suitably regular external potentials. The remainder $r^{\varepsilon}$ contains a second order finite difference of $A$, expressed in the pseudodifferential operator $\Delta[A_k]$, defined in \eqref{eq:PDO_2}, which arises due to the oscillating term $\tfrac{1}{i\varepsilon}[P,W^{\varepsilon}]$ in \eqref{eq:Wigner matrix equation_2} and which scales with $\varepsilon$. 
\begin{lemma}
\label{thm:limit remainder}
Let $A_0,A\in L^{\infty}_{\mathrm{loc}}(\mathbb{R}_t,H^{s}_{{\mathrm{loc}}}(\mathbb{R}^3_x))$ with $s\geq 1$  and $W^{\varepsilon} \in L^{\infty}(\mathbb{R}_t,L^2(\mathbb{R}^3_x \times \mathbb{R}^3_{\xi},\mathbb{M}_4(\mathbb{C})))$ uniformly in $\varepsilon$ with Wigner measure $W\in L^{\infty}(\mathbb{R}_t,L^2(\mathbb{R}^3_x \times \mathbb{R}^3_{\xi},\mathbb{M}_4(\mathbb{C})))$. Then 

\begin{enumerate}[(\alph*)]
    \item \label{eq:convergence remainder a} The remainder $r^{\varepsilon}$, defined in \eqref{eq:remainder of Wigner matrix equation},  converges to zero in $L^{\infty}(\mathbb{R}_t,\mathcal{S}'(\mathbb{R}^3_x \times \mathbb{R}^3_{\xi})) $ as $\varepsilon \rightarrow 0$.
    \item \label{eq:convergence remainder b} If in addition, $A_0,A \in H^s(\mathbb{R}^3)$, $s>5/2$, then $r^{\varepsilon}$ is bounded in $L^{\infty}([-T,T],H^{-1}(\mathbb{R}^3_x \times \mathbb{R}^3_{\xi})$.
\end{enumerate}

\end{lemma}

\begin{proof}
We introduce the notation 
\begin{align}
r^{\varepsilon} &= \varepsilon \Delta[A_k][\alpha_k,W^{\varepsilon}] +\frac{1}{2}(-\nabla_x A_k\cdot \nabla_{\xi}-\theta[A_k])[\alpha_k, W^{\varepsilon}]_+   -(-\nabla A_0 \cdot \nabla_{\xi} -\theta [A_0])W^{\varepsilon} 
\nonumber \\ &\eqqcolon  \sum_{j=1}^3 r_j^{\varepsilon} .
\end{align}
\underline{Proof of \ref{eq:convergence remainder a}}
Denote $f_k^{\varepsilon} := [\alpha_k,W^{\varepsilon}]$. Then we rewrite $r_1^{\varepsilon}$ as
$$r_1^{\varepsilon} = \varepsilon \Delta[A_k]f^{\varepsilon}_k.$$ Now let $\psi \in \mathcal{S}(\mathbb{R}^3_x\times \mathbb{R}^3_{\xi},\mathbb{M}_4(\mathbb{C}))$. Then, 
\begin{align}
    \langle \psi,\varepsilon \Delta[A_k]f_k^{\varepsilon} \rangle_{\mathcal{S}\times \mathcal{S}'} &= \int \psi(x,\xi) e^{-i(\xi-\eta)\cdot y}\frac{A_k(x+\frac{\varepsilon y}{2})-2A_k(x)+A_k(x-\frac{\varepsilon y}{2})}{2i\varepsilon} \nonumber \\ &\qquad \qquad \qquad \cdot f^{\varepsilon}_k(x,\eta) dy d\eta d x d \xi \nonumber \\
    &= \int \mathcal{F}_{\xi \rightarrow y}[{\psi}](x,y) \frac{A_k(x+\frac{\varepsilon y}{2})-2A_k(x)+A_k(x-\frac{\varepsilon y}{2})}{2i\varepsilon} \nonumber \\ &\qquad \qquad \qquad \cdot \mathcal{F}_{\eta \rightarrow y}[{f^{\varepsilon}_k}](x,-y) dy  d x ,\label{eq:integral second order difference}
\end{align}
where \begin{equation*}
    \mathcal{F}_{\eta \rightarrow y}[{f^{\varepsilon}_k}](x,-y) \rightarrow \mathcal{F}_{\eta \rightarrow y}[{f_k}](x,-y) \quad \text{ weakly in } L^2_{x,y},
\end{equation*}
since $f^{\varepsilon}_k = [\alpha_k,W^{\varepsilon}] \rightarrow [\alpha_k,W] = f_k$ weakly in $L^2_{x,\xi}$ as $\varepsilon \rightarrow 0$.

We rewrite the second order central finite difference in\eqref{eq:integral second order difference} in the following way
\begin{equation*}
    \frac{A_k(x+\frac{\varepsilon y}{2})-2A_k(x)+A_k(x-\frac{\varepsilon y}{2})}{2i\varepsilon} = \frac{A_k(x+\frac{\varepsilon y}{2})-A(x)}{2i\varepsilon} + \frac{A_k(x-\frac{\varepsilon y}{2})-A(x)}{2i\varepsilon}.
\end{equation*}
Taking the Fourier transform in $x$ yields
\begin{equation*}
    \mathcal{F}_{x\rightarrow \eta}\left[ \frac{A_k(\cdot \pm\frac{\varepsilon y}{2})-A(\cdot)}{2i\varepsilon} \right](\eta) = \frac{1}{2}\frac{e^{\pm i \frac{\varepsilon y}{2}\cdot \eta}-1}{i\varepsilon}\widehat{A}(\eta).
\end{equation*}
Then we claim that
\begin{align}
    \left\|\left(\frac{e^{\pm i \frac{\varepsilon y}{2}\cdot \eta}-1}{i\varepsilon}\pm \frac{y}{2}\cdot \eta\right)\widehat{A}(\eta) \right\|_{L^2_{\eta}} \rightarrow 0 \quad \text{ for all } y.
    \label{eq:convergence L2 fourier}
\end{align}
Indeed, denoting
\begin{equation*}
    n_{\pm}^{\varepsilon}(\eta,y) := \left|\frac{e^{\pm i \frac{i\varepsilon y}{2}\cdot \eta}-1}{i\varepsilon}\pm \frac{y}{2}\cdot \eta\right|^2|\widehat{A}(\eta)|^2,
\end{equation*}
we observe that
\begin{equation}
    n_{\pm}^{\varepsilon}(\eta,y) \leq 2\left( \left|\frac{e^{\pm i \frac{\varepsilon y}{2}\cdot \eta}-1}{\varepsilon}\right|^2 +  \left| \frac{y}{2} \cdot \eta \right|^2\right)|\widehat{A}(\eta)|^2 \leq C |y|^2 |\eta|^2|\widehat{A}(\eta)|^2,
\end{equation}
while 
\begin{equation*}
    n_{\pm}^{\varepsilon}(\eta, y) \rightarrow 0 \quad \text{ as } \varepsilon \rightarrow 0,
\end{equation*}
pointwise for all $\eta, y$. By dominated convergence we obtain the claim in \eqref{eq:convergence L2 fourier}. Thus we conlude that, for $A \in H^s$, $s\geq 1$,
\begin{equation}
    \frac{A_k(x+\frac{\varepsilon y}{2})-2A_k(x)+A_k(x-\frac{\varepsilon y}{2})}{2i\varepsilon} \rightarrow (\nabla A_k)\cdot (iy) + (\nabla A_k) \cdot (-iy) = 0 \quad \text{ in } L^2_{x} ,\label{eq:convergence L2}
\end{equation}
for all $y$ and
\begin{equation}
    \left\|\frac{A_k(x+\frac{\varepsilon y}{2})-2A_k(x)+A_k(x-\frac{\varepsilon y}{2})}{2i\varepsilon}\right\|_{L^2_x} \leq C |y|.
\end{equation}
On the other hand, since $\psi \in \mathcal{S}(\mathbb{R}^6_{x,\xi})$, we can bound
\begin{equation}
   \int \left(\int|\mathcal{F}_{\xi \rightarrow y}[\psi](x,y)|^2 |y|^2 dy\right)^{1/2} dy \leq C .
   \label{eq:bound H^1 test function}
\end{equation}
Combining \eqref{eq:convergence L2} and \eqref{eq:bound H^1 test function} we conclude that
\begin{equation}
     \langle \psi,\varepsilon \Delta[A_k]f_k^{\varepsilon} \rangle_{\mathcal{S}\times \mathcal{S}'} \rightarrow 0 \quad \text{ as } \varepsilon \rightarrow 0.
\end{equation}

Next let $$g_k^{\varepsilon} = [\alpha_k , W^{\varepsilon}]_+,$$ and write $r_2^{\varepsilon}$ as $$r_2^{\varepsilon} = \frac{1}{2}(-\nabla_x A_k\cdot \nabla_{\xi}-\theta[A_k])g_k^{\varepsilon}.$$ Let $\psi \in \mathcal{S}(\mathbb{R}^3_x\times \mathbb{R}^3_{\xi},\mathbb{M}_4(\mathbb{C}))$.
Then,
\begin{align*}
    \langle \psi, \theta[A_k] g^{\varepsilon}_k\rangle_{\mathcal{S}\times \mathcal{S}'} = \int_{\mathbb{R}^{12}} \psi(x,\xi)e^{-i(\xi-\eta)\cdot y}\frac{1}{i\varepsilon}\left(A_k(x+\frac{\varepsilon y}{2})-A_k(x-\frac{\varepsilon y}{2})\right)g^{\varepsilon}_k(x,\eta)  d \eta d y dx d \xi .
\end{align*}
By assumption, $g^{\varepsilon}_k \rightarrow g_k $ weakly in $L^2_{x,\xi}$ where
\begin{equation*}
    g_k = [\alpha_k , W]_+.
\end{equation*}
Then, since $A\in H^s_{{\mathrm{loc}}}$, $s \geq 1$, the integral above converges to
\begin{align*}
     &\int_{\mathbb{R}^{12}} \psi(x,\xi)e^{-i(\xi-\eta)\cdot y}(\nabla A_k \cdot (-iy) g_k(x,\eta)  d \eta d y dx d \xi \\
    &\qquad = \int_{\mathbb{R}^9} \widehat{\psi}(x,y) \frac{1}{i}e^{i\eta\cdot y} (\nabla A_k \cdot y) g_k(x,\eta) d \eta d y d x \\
    &\qquad = -\int_{\mathbb{R}^6}  \nabla_{\eta}{\psi}(x,\eta) \cdot\nabla A_k   g_k(x,\eta) d \eta d x \\
    &\qquad = \langle \psi, \nabla A_k \cdot \nabla_{\xi} g_k\rangle_{\mathcal{S}\times \mathcal{S}'}.
\end{align*}
Hence, $r_2^{\varepsilon}$ converges to zero in $\mathcal{S}'$.

For $r_3^{\varepsilon}$, we consider
\begin{align*}
    \langle \psi, \theta[A_0] W^{\varepsilon}\rangle_{\mathcal{S}\times \mathcal{S}'} = \int_{\mathbb{R}^{12}} \psi(x,\xi)e^{-i(\xi-\eta)\cdot y}\frac{1}{i\varepsilon}(A_0(x+\frac{\varepsilon y}{2})-A_0(x-\frac{\varepsilon y}{2}))W^{\varepsilon}(x,\eta)  d \eta d y dx d \xi.
\end{align*}
By assumption $A_0 \in H^1_{{\mathrm{loc}}}$. Then, similarly to $r_2^{\varepsilon}$, $\langle \psi, \theta[A_0] W^{\varepsilon}\rangle$ converges to
\begin{equation}
    \langle \psi \nabla A_0 \cdot \nabla_{\xi} W\rangle_{\mathcal{S}\times \mathcal{S}'}
\end{equation}
and therefore $r_3^{\varepsilon}$ converges to zero in $\mathcal{S}'$. In total, this shows the claim.

\underline{Proof of \ref{eq:convergence remainder b}} From the boundedness and continuity of the derivatives for $A,A_0\in H^s(\mathbb{R}^3)$ with $s>5/2$ we conclude that the convergence proved in \ref{eq:convergence remainder a}  holds true in $$L^{\infty}([-T,T],H^{-1}_{\mathrm{loc}}(\mathbb{R}^3_x \times \mathbb{R}^3_{\xi}).$$
Indeed, in \eqref{eq:integral second order difference}, the second order finite difference converges to zero in $L^{\infty}$ if $A$ has a bounded continuous derivative. Then we can take $f^{\varepsilon}_k$ in $H^{-1}_{\mathrm{loc}}$. By similar reasoning we conclude for $r_2^{\varepsilon}$ and $r_3^{\varepsilon}$.
\end{proof}

The following lemma guarantees regularity of the projections for given $A$.
\begin{lemma}
\label{thm:regularity projections}
    Let $A\in C^1(\mathbb{R}_t,H^s_{\mathrm{loc}}(\mathbb{R}^3_x))$ with $s\geq 1$. Then we have that
    \begin{equation*}
        \Pi_{\pm} \in C^1(\mathbb{R}_t,H^s_{\mathrm{loc}}\cap L^{\infty}(\mathbb{R}^3_x \times \mathbb{R}^3_{\xi})).
    \end{equation*}
    \begin{proof}
        Let $z =(x,\xi)$ and
        \begin{equation*}
            \Pi_{\pm}(z) = \frac{1}{2}\left(I_4 \pm \frac{P_0(z)}{\langle v(z) \rangle}\right), \quad P_0(z) = \alpha \cdot(\xi- A(x)) +\beta, \quad v(z) = \xi-A(x).
        \end{equation*}
        Therefore,
       \begin{equation*}
            |\nabla_z P_0(z)| \lesssim 1+ |\nabla A(x)|, \quad |\nabla_z v(z)| \lesssim 1+ |\nabla A(x)|,
        \end{equation*}
        and thus
        \begin{equation*}
            \nabla_z P_0(z), \, \nabla_z v(z) \in L^{2}_{\mathrm{loc}}(\mathbb{R}^6_z), 
        \end{equation*}
        as well
        \begin{equation*}
            \frac{1}{\langle v(z) \rangle}, \, \frac{v(z)}{\langle v(z) \rangle}, \, \frac{P_0(z)}{\langle v(z) \rangle} \in L^{\infty}_{\mathrm{loc}}(\mathbb{R}^6_z).
        \end{equation*}
        From \eqref{eq:derivative Pi x}-\eqref{eq:derivative Pi xi} we conclude that 
        \begin{equation*}
            \Pi_{\pm} \in L^{\infty}_{\mathrm{loc}}(\mathbb{R}^6_z), \quad \nabla_z \Pi_{\pm} (z) \in L^2_{\text{loc}}(\mathbb{R}^6_z) \quad \Rightarrow \quad  \Pi_{\pm} \in H^1_{\mathrm{loc}}(\mathbb{R}^6_z).
        \end{equation*}
        The proof easily generalizes to $H^s_{\mathrm{loc}}(\mathbb{R}^6_z)$. 
        In particular, for $s>5/2$, we obtain by Sobolev's embedding,
        \begin{equation*}
            H^s(\mathbb{R}^3_x) \hookrightarrow C^{1,\alpha}(\mathbb{R}^3_x), \, 0<\alpha <1  \quad \Rightarrow \quad \|\nabla A\|_{\infty} \leq C.
        \end{equation*}
        That is, for $A \in H^s$ with $s>5/2$ the projections $\Pi_{\pm}$ are bounded and have Hölder continuous first derivatives for some Hölder exponent $\alpha$.
    \end{proof}
\end{lemma}
\subsection{Proof of  Theorem \ref{thm:main1}}\label{sec:proofmain1}
In this section we prove the first main theorem. We first prove a general result that is of independent interest and which states the existence of a Lagrange multiplier to the constraint $[P,W]=0$.

\begin{proposition}
\label{proposition general} 

Let $W^{\varepsilon} \in \mathcal{S}'(\mathbb{R}_t \times \mathbb{R}^d_x \times \mathbb{R}^d_{\xi}))\cap L^{\infty}(\mathbb{R}_t, L^2(\mathbb{R}^d_x \times \mathbb{R}^d_{\xi}))$ with values in $\mathbb{M}_m(\mathbb{C})$  be a solution in the sense of distributions of 
\begin{equation*}
    \partial_t W^{\varepsilon} = \frac{1}{i\varepsilon}[P,W^{\varepsilon}] + \frac{1}{2}(\{P,W^{\varepsilon}\}-\{W^{\varepsilon},P\}) + r^{\varepsilon},
\end{equation*}
where $P$ is a symbol of class $C^1(\mathbb{R}_t,H^s(\mathbb{R}^d_x\times \mathbb{R}^d_{\xi}))$ with $s>5/2$  and $r^{\varepsilon} \rightarrow 0 $ in $L^{\infty}(\mathbb{R}_t, \mathcal{S}'(\mathbb{R}^d_x \times \mathbb{R}^d_{\xi}))$. Let $W$ be a limit point of $W^{\varepsilon}$ in $L^{\infty}(\mathbb{R}_t, L^2( \mathbb{R}^d_x \times \mathbb{R}^d_{\xi}))$ as $\varepsilon \to 0$. Then, there exists 
\begin{equation*}
X\in H^{-1}_{\mathrm{loc}}(\mathbb{R}_t \times \mathbb{R}^d_x \times \mathbb{R}^d_{\xi},\mathbb{M}_m(\mathbb{C})),
\end{equation*}
such that
\begin{align}\label{WignerLim}
\partial_t W &= X + \frac12\left(\{P,W\}-\{W,P\}\right),
\\
 [P,W]&=0. \label{WignerLimComm}
\end{align}
in $H^{-1}_{\mathrm{loc}}(\mathbb{R}_t \times \mathbb{R}^d_x \times \mathbb{R}^d_{\xi})$. The matrix-valued distribution $X$ satisfies the following properties:
\begin{enumerate}[(\alph*)]
    \item \label{thm:X self adjoint} $X=X^*$,
    \item \label{thm:X orthogonal} for all $\Phi\in C^1(\mathbb{R}_t,H^s( \mathbb{R}^d_x \times \mathbb{R}^d_{\xi},\mathbb{M}_m(\mathbb C)))$, $s>5/2$,
\begin{equation*}
[P,\Phi]=0\text{ on }\mathbb{R}_t \times \mathbb{R}^d_x \times \mathbb{R}^d_{\xi}\implies \tr(\Phi X)=0\text{ in }H^{-1}_{\mathrm{loc}}(\mathbb{R}_t \times \mathbb{R}^d_x \times \mathbb{R}^d_{\xi}).
\end{equation*}
\end{enumerate}
\begin{proof}
It is known from Theorem \ref{thm:wigner measures} that (possibly after extracting a subsequence $\varepsilon_n\to 0$)
\begin{equation*}
W^{\varepsilon}(t,x,\xi)=W^{\varepsilon}(t,x,\xi)^*\to W(t,x,\xi)=W(t,x,\xi)^*\geq 0\quad\text{ in }\mathcal S'(\mathbb R\times\mathbb R^d_x\times\mathbb R^d_{\xi},\mathbb{M}_m(\mathbb{C})),
\end{equation*}
and from Lemma \ref{thm:limit remainder} that
\begin{equation*}
r^{\varepsilon}\to 0\quad\text{ in }L^{\infty}(\mathbb R_t, \mathcal S'(\mathbb R^d_x\times\mathbb R^d_{\xi},\mathbb{M}_m(\mathbb{C})),
\end{equation*}
as $\varepsilon\to 0$. Therefore
\begin{equation*}
\partial_tW^{\varepsilon}\to\partial_tW \quad {\text{ in }}  H^{-1}_{\mathrm{loc}}([-T,T], \mathcal{S}'(\mathbb R^d_x\times\mathbb R^d_{\xi} ,\mathbb{M}_m(\mathbb{C}))) \quad \text{ for all } T>0
\end{equation*}
as $\varepsilon\to 0$. Moreover, since $W^{\varepsilon} \in L^{\infty}(\mathbb{R}_t,L^2(\mathbb{R}^d_x\times\mathbb R^d_{\xi},\mathbb{M}_m(\mathbb{C})))$ and $P\in C^1(\mathbb{R}_t,H^s_{\mathrm{loc}}(\mathbb{R}^d_x\times \mathbb{R}^d_{\xi},\mathbb{M}_m(\mathbb{C})))$ with $s>5/2$ we obtain that
\begin{equation*}   
\begin{aligned}&\{P,W^{\varepsilon}\}=\partial_{\xi_j}P\partial_{x_j}W^{\varepsilon}-\partial_{x_j}P\partial_{\xi_j}W^{\varepsilon}\to\partial_{\xi_j}P\partial_{x_j}W-\partial_{x_j}P\partial_{\xi_j}W=\{P,W\},
\\
&\{W^{\varepsilon},P\}=\partial_{\xi_j}W^{\varepsilon}\partial_{x_j}P-\partial_{x_j}W^{\varepsilon}\partial_{\xi_j}P\to\partial_{\xi_j}W\partial_{x_j}P-\partial_{x_j}W\partial_{\xi_j}P=\{W,P\},
\end{aligned}
\end{equation*}
in $L^{\infty}(\mathbb{R}_t,H^{-1}_{\mathrm{loc}}(\mathbb{R}_x^d \times \mathbb{R}^d_{\xi}, \mathbb{M}_m(\mathbb{C})))$ as $\varepsilon\to 0$. Therefore, we obtain 
\begin{equation}
    W^{\varepsilon} \rightarrow W \quad {\text{ in }}  H^{-1}_{\mathrm{loc}}([-T,T] \times \mathbb R^d_x\times\mathbb R^d_{\xi} ,\mathbb{M}_m(\mathbb{C})) \quad \text{ for all } T>0.\label{eq:convergence W H-1}
\end{equation}
Besides
\begin{equation*}
\begin{aligned}
\tfrac1{i\varepsilon}[P(t,x,\xi),W^{\varepsilon}(t,x,\xi)]=&-r^{\varepsilon}(t,x,\xi)-\partial_tW^{\varepsilon}(t,x,\xi)
\\
&-\tfrac12\left(\{P(t,x,\xi),W^{\varepsilon}(t,x,\xi)\}-\{W^{\varepsilon}(t,x,\xi),P(t,x,\xi)\}\right)
\end{aligned}
\end{equation*}
converges to a limit in $\mathcal S'(\mathbb R_t\times\mathbb R^d_x\times\mathbb R^d_{\xi} ,\mathbb{M}_m(\mathbb{C}))$ as $\varepsilon\to 0$. 

Denoting by $X\in\mathcal S'(\mathbb R_t\times\mathbb R^d_x\times\mathbb R^d_{\xi} ,\mathbb{M}_m(\mathbb{C}))$ this limit, we find that
\begin{equation*}
\partial_tW+\frac12\left(\{P,W\}-\{W,P\}\right)+X=0\quad\text{ in }\mathcal S'(\mathbb R_t\times\mathbb R^d_x\times\mathbb R^d_{\xi} ,\mathbb{M}_m(\mathbb{C})).
\end{equation*}
Owing to \eqref{eq:convergence W H-1} we deduce that in fact
\begin{equation}
    X \in H^{-1}_{\mathrm{loc}}([-T,T]\times \mathbb R^d_x\times\mathbb R^d_{\xi} ,\mathbb{M}_m(\mathbb{C})) \quad \text{ for all } T>0.
\end{equation}
Besides
\begin{equation}\label{[P,Weps]to0}
\begin{aligned}
{}[P(t,x,\xi),W^{\varepsilon}(t,x,\xi)]=&-i\varepsilon r^{\varepsilon}(t,x,\xi)-i\varepsilon\partial_tW^{\varepsilon}(t,x,\xi)
\\
&-\tfrac12i\varepsilon\left(\{P(t,x,\xi),W^{\varepsilon}(t,x,\xi)\}-\{W^{\varepsilon}(t,x,\xi),P(t,x,\xi)\}\right)
\\
&\to 0
\end{aligned}
\end{equation}
in $\mathcal S'(\mathbb R_t\times\mathbb R^d_x\times\mathbb R^d_{\xi} ,\mathbb{M}_m(\mathbb{C}))$ as $\varepsilon\to 0$, while
\begin{equation*}
\begin{aligned}
{}[P(t,x,\xi),W^{\varepsilon}(t,x,\xi)]=&P(t,x,\xi)W^{\varepsilon}(t,x,\xi) - W^{\varepsilon}(t,x,\xi)P(t,x,\xi)
\\
&\to P(t,x,\xi)W(t,x,\xi)-W(t,x,\xi)P(t,x,\xi)
\\
&\,\,\quad=[P(t,x,\xi),W(t,x,\xi)]
\end{aligned}
\end{equation*}
in $ H^{-1}_{\mathrm{loc}}([-T,T]\times \mathbb R^d_x\times\mathbb R^d_{\xi} ,\mathbb{M}_m(\mathbb{C}))$ for all $ T>0$ as $\varepsilon\to 0$ since $P$ is of class $C^1_tH^s_{\mathrm{loc}}$, $s>5/2$.
By \eqref{[P,Weps]to0} and uniqueness of the limit in $\mathcal S'(\mathbb R_t\times\mathbb R^d_x\times\mathbb R^d_{\xi})$, we conclude that
\begin{equation*}
[P,W]=0\qquad\text{ in } H^{-1}_{\mathrm{loc}}([-T,T]\times \mathbb R^d_x\times\mathbb R^d_{\xi} ,\mathbb{M}_m(\mathbb{C})) \quad \text{ for all } T>0,
\end{equation*}
This proves the statement about the convergence. It remains to prove \ref{thm:X self adjoint} and \ref{thm:X orthogonal}. For the former, observe that
\begin{equation*}
    -i[P,W^{\varepsilon}]= \left(-i[P,W^{\varepsilon}]\right)^{\ast},  
\end{equation*}
since $P=P^{\ast}$ and $W^{\varepsilon} = (W^{\varepsilon})^{\ast}$ and thus we deduce $X=X^{\ast}$. For the latter we observe that for $\Phi$ satisfying $[P,\Phi]=0$,
\begin{equation}
    \tr_{\mathbb{C}^4}(\tfrac{1}{i\varepsilon}[P,W^{\varepsilon}]\Phi) = \tr_{\mathbb{C}^4}(\tfrac{1}{i\varepsilon}W^{\varepsilon}[P,\Phi])  = 0
\end{equation} 
due to the cyclicity of the trace. Moreover, we have that $\tr_{\mathbb{C}^4}(\tfrac{1}{i\varepsilon}[P,W^{\varepsilon}]\Phi)$ converges to $\tr_{\mathbb{C}^4}(-iX\Phi)$, which proves the claim.
\end{proof}
\end{proposition}
In particular
\begin{equation*}
\mathrm{tr}(X)=\mathrm{tr}(PX)=\ldots=\mathrm{tr}(P^nX)=0\quad\text{ for all }n\ge 0.
\end{equation*}

Notice however that $X\in H^{-1}_{\mathrm{loc}}([-T,T]\times \mathbb R^d_x\times\mathbb R^d_{\xi} ,\mathbb{M}_m(\mathbb{C}))$, so that Proposition \ref{proposition general} does not imply in general that $W|_{t=0}$
is well-defined. More information is needed in order to formulate the Cauchy problem for the equation \eqref{WignerLim}-\eqref{WignerLimComm}. In the sequel, we shall discuss the special case of the Dirac equation on $\mathbb{R}^3$, i.e. we take
\begin{equation}
    P(t,x,\xi) = \alpha\cdot(\xi-A(t,x)) + \beta -A_0(t,x)I_4, \quad (t,x,\xi) \in \mathbb{R}_t \times \mathbb{R}^3_x \times \mathbb{R}^3_{\xi},
\end{equation}
with $A$ having bounded derivatives and we set $m=4$ in $\mathbb{M}_m(\mathbb{C})$. In that case we have the following Lemma.
\begin{lemma}
Assume that $A \in C^1(\mathbb{R}_t,H^s_{\mathrm{loc}}(\mathbb R^3_x))$ with $s>5/2$. Then
\begin{enumerate}[(\alph*)]
    \item \label{ProjectedX} One has
\begin{equation}
\Pi_+X\Pi_+=\Pi_-X\Pi_-=0 \quad\text{ in }H^{-1}_{\mathrm{loc}}([-T,T]\times\mathbb{R}_x^3\times\mathbb{R}_{\xi}^3, \mathbb{M}_4(\mathbb C)) \quad \text{ for all } T>0.
\end{equation}
\item There exists $Y=Y^*\in H^{-1}_{\mathrm{loc}}([-T,T]\times \mathbb R^d_x\times\mathbb R^d_{\xi} ,\mathbb{M}_4(\mathbb{C}))$ such that
\begin{equation}
X=-i[P,Y],\quad\text{ with }\Pi_+Y\Pi_+=\Pi_-Y\Pi_-=0. \label{eq:distributionY}
\end{equation}
\end{enumerate} 
\end{lemma}
\begin{proof}
Let $\Psi\in \mathcal{S}( \mathbb{R}_t \times \mathbb{R}^3_x \times \mathbb{R}^3_{\xi},\mathbb{M}_4(\mathbb C))$, and set
\begin{equation*}
\Phi:=\Pi_{\pm}\Psi\Pi_{\pm}\in C^1(\mathbb{R}_t,H^1(\mathbb{R}_x^3\times\mathbb{R}_{\xi}^3,\mathbb{M}_4(\mathbb C))).
\end{equation*}
Indeed, since the eigenprojections of $P$ satisfy $\Pi_\pm\in C^1(\mathbb{R}_t,H^s_{\mathrm{loc}}(\mathbb{R}_x^3\times\mathbb{R}_{\xi}^3,\mathbb{M}_4(\mathbb C)))$, $s>5/2$, we have (denoting $z=(x,\xi)$),
\begin{align*}
    \nabla_z \Phi = (\nabla_z\Pi_{\pm}) \Psi \Pi_{\pm} + \Pi_{\pm} \Psi (\nabla_z\Pi_{\pm}) + \Pi_{\pm} (\nabla_z \Psi) \Pi_{\pm}.
\end{align*}
By Lemma \ref{thm:regularity projections} we obtain that
\begin{equation}
    \nabla_z \Pi_{\pm}  \in L^{\infty}_{\mathrm{loc}} \quad \Rightarrow \quad \nabla_z \Phi \in L^2.
\end{equation}
Since $\Pi_+\Pi_-=\Pi_-\Pi_+=0$ while $\Pi_{\pm}^2=\Pi_{\pm}$, one has obviously
\begin{equation*}
\begin{aligned}
{}[P,\Phi]=[\lambda_+\Pi_+ + \lambda_- \Pi_-,\Phi]=&\langle\xi-A\rangle[\Pi_+-\Pi_-,\Pi_{\pm}\Psi\Pi_{\pm}]
\\
=&\pm\langle\xi-A\rangle(\Pi_{\pm}^2\Psi\Pi_{\pm}-\Pi_{\pm}\Psi\Pi_{\pm}^2)=0.
\end{aligned}
\end{equation*}
By Proposition \ref{proposition general}
\begin{equation*}
\tr(\Phi X)=\tr(\Pi_{\pm}\Psi\Pi_{\pm}X)=\tr(\Psi\Pi_{\pm}X\Pi_{\pm})=0
\end{equation*}
in $H^{-1}_{\mathrm{loc}}([-T,T]\times\mathbb{R}_x^3\times\mathbb{R}_{\xi}^3)$, and since this holds for all 
$\Psi\in\mathcal S(\mathbb{R}_t\times\mathbb{R}_x^3\times\mathbb{R}_{\xi}^3, \mathbb{M}_4(\mathbb C))$, it implies that
\begin{equation*}
\Pi_{\pm}X\Pi_{\pm}=0\quad\text{ in }H^{-1}_{\mathrm{loc}}([-T,T]\times\mathbb{R}_x^3\times\mathbb{R}_{\xi}^3, \mathbb{M}_4(\mathbb C)) \quad \text{ for all } T>0.
\end{equation*}

Now seek $Y \in H^{-1}_{\mathrm{loc}}([T,T]\times \mathbb{R}^3_x \times \mathbb{R}^3_{\xi},\mathbb{M}_4(\mathbb{C}))$ such that
\begin{equation*}
\Pi_{\pm}Y\Pi_{\pm} \in H^{-1}_{\mathrm{loc}}([T,T]\times \mathbb{R}^3_x \times \mathbb{R}^3_{\xi},\mathbb{M}_4(\mathbb{C})), \quad \Pi_{\pm}Y\Pi_{\pm} =0, \quad \text{ for all } T>0.
\end{equation*}
Hence
\begin{equation*}
Y=(\Pi_++\Pi_-)Y(\Pi_++\Pi_-)=\Pi_-Y\Pi_++\Pi_+Y\Pi_-,
\end{equation*}
so that
\begin{equation*}
\begin{aligned}
-i[P,Y] = -i[\lambda_+\Pi_++\lambda_- \Pi_-,Y] =&-i\langle\xi-A\rangle[\Pi_+-\Pi_-,\Pi_-Y\Pi_++\Pi_+Y\Pi_-]
\\
=&-i\langle\xi-A\rangle(-\Pi_-^2Y\Pi_+-\Pi_-Y\Pi_+^2+\Pi_+^2Y\Pi_-+\Pi_+Y\Pi_-^2)
\\
=&2i\langle\xi-A\rangle(\Pi_-Y\Pi_+-\Pi_+Y\Pi_-).
\end{aligned}
\end{equation*}
By \ref{ProjectedX}, one has
\begin{equation*}
X=\Pi_-X\Pi_++\Pi_+X\Pi_-,
\end{equation*}
by analogy with the case of $Y$ as explained above, so that
\begin{equation*}
-i[P,Y]=X\iff 2i\langle\xi-A\rangle(\Pi_-Y\Pi_+-\Pi_+Y\Pi_-)=\Pi_-X\Pi_++\Pi_+X\Pi_-.
\end{equation*}
This last equation is solved by setting
\begin{equation*}
\Pi_{\mp}Y\Pi_{\pm}=\frac{\Pi_{\mp}X\Pi_{\pm}}{2i\langle\xi-A\rangle},
\end{equation*}
i.e.
\begin{equation*}
Y:=\frac{\Pi_-X\Pi_+-\Pi_+X\Pi_-}{2i\langle\xi-A\rangle}\in H^{-1}_{\mathrm{loc}}([T,T]\times \mathbb{R}^3_x \times \mathbb{R}^3_{\xi},\mathbb{M}_4(\mathbb{C})) \quad \text{ for all } T>0.
\end{equation*}
This expression is well-defined  since, for $s>5/2$,
\begin{equation*}
\frac{1}{\langle\xi-A\rangle}, \,\Pi_\pm\in C^1(\mathbb{R}_t,W^{1,\infty}_{\mathrm{loc}}(\mathbb{R}_x^3\times\mathbb{R}_{\xi}^3, \mathbb{M}_4(\mathbb C)))).
\end{equation*}
 Observe that the definition of $Y$ above implies that
\begin{equation*}
Y=Y^*,\quad\text{ since }X=X^*.
\end{equation*}
This completes the proof.
\end{proof}
Therefore we can rewrite \eqref{WignerLim}-\eqref{WignerLimComm} as
\begin{align}
\partial_t W &= -i[P,Y] + \frac12\left(\{P,W\}-\{W,P\}\right), \label{WignerLimY}
\\
 [P,W]&=0. \label{WignerLimCommY}
\end{align}
In order to discuss the Cauchy problem for \eqref{WignerLimY}-\eqref{WignerLimCommY} we need the following auxilliary lemma.
\begin{lemma}
    Let $S=S^*\in \mathbb{M}_4(\mathbb C)$ satisfy $S^2=I_4$, with eigenvalues $\pm 1$ and associated eigenprojections
\begin{equation*}
\Pi_\pm:=\tfrac12(I_4\pm S)=\Pi_\pm^*=\Pi_\pm^2.
\end{equation*}
Let $A\in M_4(\mathbb C)$; then
\begin{equation*}
\Pi_+A\Pi_+=\Pi_-A\Pi_-=0\implies[S,[S,A]]=4A.
\end{equation*}
\begin{proof}As explained above
\begin{equation*}
A=(\Pi_++\Pi_-)A(\Pi_++\Pi_-)=\Pi_-A\Pi_++\Pi_+A\Pi_-,
\end{equation*}
so that
\begin{equation*}
\begin{aligned}
{}[S,[S,A]]=&S^2A+AS^2-2SAS=2(A-SAS)
\\
=&2(\Pi_-A\Pi_++\Pi_+A\Pi_--(\Pi_+-\Pi_-)A(\Pi_+-\Pi_-))
\\
=&4(\Pi_-A\Pi_++\Pi_+A\Pi_-)=4A.
\end{aligned}
\end{equation*}
This completes the proof.
\end{proof}
\end{lemma}
In the next lemma we characterize the distribution $Y$ and show that it is sensible to speak of $W$ at time $t=0$. Here we exploit the analogy with the Navier-Stokes equation in Section \ref{section analogy navier stokes} in order to express $Y$ in terms of $P$ and $W$, similarly to how the pressure $p$ is expressed in terms of $u$ and its derivatives.
\begin{lemma}
\label{thm:characterization Y}
Assume that $A\in C^1(\mathbb R_t,H^s(\mathbb R^3_x))$, $s >5/2$, and let  $W=W^*\ge 0$ belonging to $ L^\infty(\mathbb R_t;L^2(\mathbb R^3_x \times \mathbb{R}^3_{\xi},\mathbb{M}_4(\mathbb C)))$ satisfy \eqref{WignerLim}-\eqref{WignerLimComm}. Then
\begin{enumerate}[(\alph*)]
    \item \label{FormulaY}The matrix-valued distribution $Y$ \eqref{eq:distributionY} satisfying $\Pi_+Y\Pi_+=\Pi_-Y\Pi_-=0$ is given by the formula
\begin{equation}
Y=\frac{i\left([\partial_tP,W]-\tfrac12[P,\{P,W\}-\{W,P\}]\right)}{4\langle\xi-A\rangle^2}. \label{eq:Y}
\end{equation}
\item \label{YBounded}For each $\chi\equiv\chi(x,\xi)\in C^\infty_c(\mathbb R^3_x\times\mathbb R^3_\xi)$ and each $T>0$, the matrix-valued distribution $Y$ \eqref{eq:distributionY} satisfies
\begin{equation*}
\chi Y\in L^\infty([-T,T];H^{-1}(\mathbb R^3_x\times\mathbb R^3_\xi,\mathbb{M}_4(\mathbb C))).
\end{equation*}
\item \label{dtWBounded} Similarly,
\begin{equation*}
\partial_t(\chi W)\in L^\infty([-T,T];H^{-1}(\mathbb R^3_x\times\mathbb R^3_\xi,\mathbb{M}_4(\mathbb C))).
\end{equation*}
\end{enumerate}
In particular, for each $\chi\equiv\chi(x)\in C^\infty_c(\mathbb R^3_x\times\mathbb R^3_\xi)$, the matrix-valued function
\begin{equation*}
(t,x,\xi)\mapsto\chi(x,\xi)W(t,x,\xi)
\end{equation*}
belongs to $C(\mathbb R_t;H^{-1}(\mathbb R^3_x\times\mathbb R^3_\xi,\mathbb{M}_4(\mathbb C)))$, so that the trace of $W|_{t=0}$ is well-defined as an element of $H^{-1}_{\mathrm{loc}}(\mathbb R^3_x\times\mathbb R^3_\xi,\mathbb{M}_4(\mathbb C))$.
\begin{proof}
\underline{Proof of \ref{FormulaY}:} Let us apply $-i[P,\cdot]$ to both sides of equation \eqref{WignerLimY}:
\begin{equation*}
\begin{aligned}
{}4\langle\xi-A\rangle^2Y=[P,[P,Y]]=&-i[P,\partial_tW]-\tfrac12i[P,\{P,W\}]+\tfrac12i[P,\{W,P\}]
\\
=&i[\partial_tP,W]-\tfrac12i[P,\{P,W\}]+\tfrac12i[P,\{W,P\}],
\end{aligned}
\end{equation*}
since
\begin{equation*}
[P,W]=0\implies[P,\partial_tW]=\partial_t[P,W]-[\partial_tP,W]=-[\partial_tP,W].
\end{equation*}
Therefore
\begin{equation*}
Y=\frac{i\left([\partial_tP,W]-\tfrac12[P,\{P,W\}-\{W,P\}]\right)}{4\langle\xi-A\rangle^2}.
\end{equation*}
\underline{Proof of \ref{YBounded}:} We have that
\begin{equation*}
\partial_tP=-\alpha \cdot(\partial_t A) - \partial_t A_0 I_4,
\end{equation*}
while
\begin{equation*}
\begin{aligned}
\{P,W\}-\{W,P\}
=&\sum_{j=1}^3(\partial_{x_j}+\nabla_xA_j\cdot\nabla_\xi)(\alpha_jW+W\alpha_j)
- 2\sum_{j=1}^3 \partial_{x_j} A_0 \partial_{\xi_j} W.
\end{aligned}
\end{equation*}
Therefore, assuming that $W\in L^\infty(\mathbb R;L^2(\mathbb R^3_x\times\mathbb R^3_\xi,\mathbb{M}_4(\mathbb C)))$ implies that
\begin{equation}\label{RegPW-WP}
\begin{aligned}
\{P,W\}-\{W,P\}
\in L^\infty([-T,T];H^{-1}_{\mathrm{loc}}(\mathbb R^3_x\times\mathbb R^3_\xi;\mathbb{M}_4(\mathbb C)))
\end{aligned}
\end{equation}
as in Proposition \ref{proposition general}.
By the same token
\begin{equation*}
\begin{aligned}
{}[\partial_tP,W]=&-\sum_{j=1}^3\partial_tA_j[\alpha_j,W]\in L^\infty([-T,T];L^2(\mathbb R^3_x\times\mathbb R^3_\xi;\mathbb{M}_4(\mathbb C))),
\\
{}[P,\{P,W\}-\{W,P\}]=&\sum_{j=1}^3(\partial_{x_j}+\nabla_xA_j\cdot\nabla_\xi)[\beta,\alpha_jW+W\alpha_j]
\\
&+\sum_{j,k=1}^3(\xi_k-A_k)(\partial_{x_j}+\nabla_xA_j\cdot\nabla_\xi)[\alpha_k,\alpha_jW+W\alpha_j]
\\
&\qquad\in L^\infty([-T,T];H^{-1}_{\mathrm{loc}}(\mathbb R^3_x\times\mathbb R^3_\xi;\mathbb{M}_4(\mathbb C))).
\end{aligned}
\end{equation*}
Therefore
\begin{equation*}
\chi Y=\frac{i\chi\left([\partial_tP,W]-\tfrac12[P,\{P,W\}-\{W,P\}]\right)}{4\langle\xi-A\rangle^2}
\in L^\infty(\mathbb R_t;H^{-1}(\mathbb R^3_x\times\mathbb R^3_\xi,\mathbb{M}_4(\mathbb C))).
\end{equation*}
\underline{Proof of \ref{dtWBounded}} By \eqref{WignerLimY} and \eqref{RegPW-WP}, one computes
\begin{equation*}
\begin{aligned}
\partial_t(\chi W)=&-\frac12\chi\left(\{P,W\}-\{W,P\}\right)+i[P,\chi Y]
\\
=&-\frac12\sum_{j=1}^3(\partial_{x_j}+\nabla_xA_j\cdot\nabla_\xi)(\alpha_j\chi W+\chi W\alpha_j)
\\
&+\frac12\sum_{j=1}^3(\alpha_jW+W\alpha_j)(\partial_{x_j}\chi+\nabla_xA_j\cdot\nabla_\xi\chi)
\\
&+i[P,\chi Y]\in L^\infty(\mathbb R_t;H^{-1}(\mathbb R^3_x\times\mathbb R^3_\xi,\mathbb{M}_4(\mathbb C))),
\end{aligned}
\end{equation*}
for each $\chi\in C^\infty_c(\mathbb R^3_x\times\mathbb R^3_\xi)$.
\end{proof}
\end{lemma}
Following Lemma \ref{thm:characterization Y}, we can now pose the Cauchy problem
\begin{align}
\label{WignerLim3}
\partial_tW &=-i[P,Y]+\frac12\left(\{P,W\}-\{W,P\}\right),
\\
[P,W]&=0,\quad\Pi_+Y\Pi_+=\Pi_-Y\Pi_-=0,
\\
W|_{t=0}&=W^\mathrm{in}, \label{WignerLim3Data}
\end{align}
with $Y$ given by \eqref{eq:Y}. It remains to identify the initial data $W^\mathrm{in}$. By assumption \eqref{eq:convergence initial data} in Theorem \ref{thm:main1} we have that
\begin{equation*}
W_{\pm}^{\mathrm{in},\varepsilon} := \Pi_{\pm} W^{\mathrm{in},\varepsilon} \Pi_{\pm} \underset{\varepsilon \rightarrow 0}{\longrightarrow} \Pi_{\pm} W^{\mathrm{in}} \Pi_{\pm} =: W^{\mathrm{in}}_{\pm}, 
\end{equation*}
weakly in $L^2(\mathbb{R}^3_x \times \mathbb{R}^3_{\xi})$. Moreover, by assumption \eqref{eq:assumption lambda} and Plancherel's theorem we have that
\begin{equation}
    \|W^{\varepsilon} (t) \|_{L^2_{x,\xi}}^2  = \|W_{\mathrm{in}}^{\varepsilon}\|_{L^2_{x,\xi}}^2 \leq C.
\end{equation}
\begin{lemma}
\label{thm:initial data}
The solution $W \in L^{\infty}(\mathbb{R}_t, L^2(\mathbb{R}^3_x\times \mathbb{R}^3_{\xi}))$ of \eqref{WignerLim3} satisfies
\begin{equation*}
W=\Pi_+W\Pi_++\Pi_-W\Pi_-,
\end{equation*}
Moreover, let $W^{\mathrm{in}}_{\pm}$ be the weak limit of $W_{\pm}^{\mathrm{in},\varepsilon}$ in $L^2(\mathbb{R}^3_x\times \mathbb{R}^3_{\xi})$. Then the initial data \eqref{WignerLim3Data} for \eqref{WignerLim3} is given by
\begin{equation*}
W^{\mathrm{in}}:=W_+^{\mathrm{in}}+W_-^{\mathrm{in}}.
\end{equation*}

\begin{proof} First we observe that
\begin{equation*}
\Pi_+W\Pi_-=\Pi_-W\Pi_+=0\quad\text{ so that }W=\Pi_+W\Pi_++\Pi_-W\Pi_-.
\end{equation*}
Indeed
\begin{equation*}
\begin{aligned}
0=\Pi_+[P,W]\Pi_- &= \Pi_+[\lambda_+ \Pi_+ + \lambda_- \Pi_-,W]\Pi_-=(\lambda_+ - \lambda_-)\Pi_+W\Pi_-
\\ &= 2\langle\xi-A\rangle\Pi_+W\Pi_-.
\end{aligned}
\end{equation*}
The proof for $\Pi_- W \Pi_+$ is similar.
Compressing both sides of \eqref{eq:Wigner matrix equation_2} by $\Pi_\pm$ shows that
\begin{equation*}
\begin{aligned}
\partial_t(\Pi_\pm W^{\varepsilon}\Pi_\pm) = \frac12\Pi_\pm\left(\{P,W^{\varepsilon}\}-\{W^{\varepsilon},P\}\right)\Pi_\pm + \frac1{i{\varepsilon}}\Pi_\pm[P,W^{\varepsilon}]\Pi_\pm + \Pi_\pm r^{\varepsilon}\Pi_\pm
\\
-(\partial_t\Pi_\pm)W^{\varepsilon}\Pi_\pm -\Pi_\pm W^{\varepsilon}(\partial_t\Pi_\pm)&.
\end{aligned}
\end{equation*}
As above,
\begin{equation*}
\begin{aligned}
\Pi_\pm[P,W^{\varepsilon}]\Pi_\pm= \Pi_\pm[\lambda_+ \Pi_+ + \lambda_- \Pi_-,W^{\varepsilon}]\Pi_\pm = \lambda_{\pm}\Pi_{\pm}W^{\varepsilon} \Pi_{\pm}-\lambda_{\pm}\Pi_{\pm}W^{\varepsilon} \Pi_{\pm}
=0.
\end{aligned}
\end{equation*}
Next we use assumption \eqref{eq:assumption lambda}, i.e.
\begin{equation*}
\sup_{{\varepsilon}>0}\|W^{\varepsilon}\|_{L^\infty([-T,T],L^2_{x,\xi})}<+\infty.
\end{equation*}
We have that
\begin{equation*}
\begin{aligned}
\sup_{{\varepsilon}>0}\|(\partial_t\Pi_\pm)W^{\varepsilon}\Pi_\pm\|_{L^\infty([-T,T],L^2_{x,\xi})}<+\infty,
\\
\sup_{{\varepsilon}>0}\|\Pi_\pm W^{\varepsilon}(\partial_t\Pi_\pm)\|_{L^\infty([-T,T],L^2_{x,\xi})}<+\infty,
\end{aligned}
\end{equation*}
for all $T>0$ since $\Pi_\pm\in C^1(\mathbb{R}_t,H^s(\mathbb{R}_x^3\times\mathbb{R}_{\xi}^3))$ for $s>5/2$. Proceeding as in \eqref{RegPW-WP} shows that
\begin{equation*}
\chi(x,\xi)\left(\{P(t,x,\xi),W^{\varepsilon}(t,x,\xi)\}-\{W^{\varepsilon}(t,x,\xi),P(t,x,\xi)\}\right)
\end{equation*}
is bounded in $L^{\infty}([-T,T] ,H^{-1}(\mathbb R^3_x\times\mathbb R^3_\xi;\mathbb{M}_4(\mathbb C)))$, and therefore
\begin{equation*}
\sup_{{\varepsilon}>0}\|\chi\Pi_\pm\left(\{P,W^{\varepsilon}\}\!-\!\{W^{\varepsilon},P\}\right)\Pi_\pm\|_{L^{\infty}([-T,T],H^{-1} (\mathbb{R}^6_{x,\xi}))}<+\infty,
\end{equation*}
for all $T>0$, since $\Pi_\pm\in C^1(\mathbb{R}_t,H^s(\mathbb{R}_x^3\times\mathbb{R}_{\xi}^3))$, $s>5/2$, and $\chi\in C^\infty_c(\mathbb{R}_x^3\times\mathbb{R}_{\xi}^3)$.
Finally, by Lemma \ref{thm:limit remainder},
\begin{equation*}
\|\chi\Pi_\pm r^{\varepsilon}\Pi_\pm\|_{L^{\infty}([-T,T],H^{-1}(B_R(0)^2))}=O({\varepsilon}),
\end{equation*}
since $\Pi_\pm\in C^1(\mathbb{R}_t,H^s(\mathbb{R}_x^3\times\mathbb{R}_{\xi}^3))$ for $s>5/2$ and $\chi\in C^\infty_c(\mathbb R_x^3\times\mathbb R_{\xi}^3)$.

Hence
\begin{equation*}
\sup_{0<{\varepsilon}<1}\|\partial_t(\chi\Pi_\pm W\Pi_\pm )\|_{L^{\infty}([-T,T],H^{-1}(\mathbb{R}^6_{x,\xi}))}<+\infty,
\end{equation*}
so that (by Ascoli's theorem)
\begin{equation*}
\Pi_\pm W^{\varepsilon}(t,\cdot,\cdot)\Pi_\pm \to \Pi_\pm W(t,\cdot,\cdot)\Pi_\pm \text{ in }H^{-1}_{\mathrm{loc}}(\mathbb R^3\times\mathbb R^3)\text{ uniformly in }t\in[-T,T].
\end{equation*}
\end{proof}
\end{lemma}

Finally, we conclude the proof of Theorem \ref{thm:main1} by discussing the convergence of the densities \eqref{eq:charge current densities}.
We recall that $W=W^*\ge 0$. Assume that
\begin{equation}
\label{InitProba}
\iint_{\mathbb R^3\times\mathbb R^3}\left(\tr(W_+^{\mathrm{in}})+\tr(W_-^{\mathrm{in}})\right)dxd\xi=1.
\end{equation}

We shall need the following (classical) lemma, of independent interest.

\begin{lemma} Let $0\leq \rho(t,z)\equiv\rho \in C([0,T];H^{-m}(\mathbb R^d_z))\cap L^\infty([0,T];L^1(\mathbb R^d_z))$ satisfy
\begin{equation*}
\left\{
\begin{aligned}
{}&\partial_t\rho(t,z)+\nabla_z\cdot J(t,z)=0,
\\
&\rho|_{t=0}=\rho^{\mathrm{in}},
\end{aligned}\right.
\end{equation*}
where $J\in L^1([0,T]\times\mathbb R^d_z)$ and 
\begin{equation*}
0\le\rho^{\mathrm{in}}\text{ a.e. satisfying }\quad\int_{\mathbb R^d}\rho^{\mathrm{in}}(z)dz=1.
\end{equation*}
Then
\begin{equation*}
\int_{\mathbb R^d}\rho(t,z)dz=1\quad\text{ for a.e. }t\in[0,T].
\end{equation*}

\begin{proof}
Let $\chi\in C^\infty_c(\mathbb R^3\times\mathbb R^3)$ such that
\begin{equation*}
0\le\mathbf 1_{B(0,1)}\le\chi\le\mathbf 1_{B(0,2)}\quad\text{ and }\quad\|\nabla\chi\|_{L^\infty(\mathbb R^d)}\le 2.
\end{equation*}
Then
\begin{equation*}
\frac{d}{dt}\int_{\mathbb R^d}\rho(t,z)\chi(z/n)dz=\frac1n\int_{\mathbb R^d}J(t,z)\cdot\nabla\chi(z/n)dz.
\end{equation*}
In particular, there exists a negligible subset $N\subset(0,T]$ such that
\begin{equation*}
\|\rho(t,\cdot)\|_{L^1\mathbb R^d)}\le C<\infty,\qquad t\in[0,T]\setminus N.
\end{equation*}
For all $t\in[0,T]\setminus N$, by dominated convergence 
\begin{equation*}
\int_{\mathbb R^d}\rho(t,z)\chi(z/n)dz\to\int_{\mathbb R^d}\rho(t,z)dz\text{ as }n\to\infty,
\end{equation*}
and likewise
\begin{equation*}
\int_{\mathbb R^d}\rho^{\mathrm{in}}(z)\chi(z/n)dz\to\int_{\mathbb R^d}\rho^{\mathrm{in}}(z)dz=1.
\end{equation*}
Besides, enlarging the negligible set $N$ if needed, 
\begin{equation*}
\int_{\mathbb R^d}\rho(t,z)\chi(z/n)dz=\int_{\mathbb R^d}\rho^{\mathrm{in}}(z)\chi(z/n)dz+\frac1n\int_0^t\int_{\mathbb R^d}J(s,z)\cdot\nabla\chi(z/n)dzds,
\end{equation*}
for all $t\in[0,T]\setminus N$. Therefore
\begin{equation*}
\left|\int_{\mathbb R^d}\rho(t,z)\chi(z/n)dz-\int_{\mathbb R^d}\rho^{\mathrm{in}}(z)\chi(z/n)dz\right|\le\frac2n\int_0^T\int_{\mathbb R^d}|J(s,z)|dzds
\end{equation*}
and passing to the limit as $n\to\infty$ leads to the announced conclusion.
\end{proof}
\end{lemma}

\begin{lemma}
Under the assumptions of Theorem \ref{thm:main1}, together with \eqref{InitProba}, one has
\begin{equation*}
\iint_{\mathbb R^6_{x,\xi}}\tr(W)(t,x,\xi)dxd\xi=1\quad\text{ for a.e. }t\ge 0.
\end{equation*}
In particular, for a.e. $t\ge 0$, one has the convergence of the densities \eqref{eq:charge current densities}:
\begin{align*}
\sum_{j=1}^{\infty}\lambda^\varepsilon_j|\psi^\varepsilon_j(t,\cdot)|^2 &\to\rho(t,\cdot)=\int_{\mathbb R^3_\xi}\tr(W(t,\cdot,\xi))d\xi, \\
\sum_{j=1}^{\infty}\lambda^\varepsilon_j\langle \psi^\varepsilon_j(t,\cdot),\alpha_k \psi^\varepsilon_j(t,\cdot)\rangle_{\mathbb{C}^4} &\to J_k(t,\cdot)=\int_{\mathbb R^3_\xi}\tr(\alpha_k W(t,\cdot,\xi))d\xi,
\end{align*}
narrowly on $\mathbb R^3\text{ as }\varepsilon\to 0$ and
\begin{equation*}
\sup_{0< \varepsilon \leq 1}\sum_{j=1}^{\infty}\lambda^\varepsilon_j\int_{|\xi|>R/\varepsilon}\left|\widehat{\chi\psi^\varepsilon_j(t,\xi)}\right|^2d\xi \longrightarrow 0 \quad \text{ as } R\rightarrow \infty,
\end{equation*}
for all $\chi\in C_c(\mathbb R^3)$.
\begin{proof}
Taking the trace of both sides of the equation in \eqref{WignerLim3} shows that
\begin{equation*}
\partial_t\tr(W)+\sum_{j=1}^3(\partial_{x_j}+\nabla_xA_j\cdot\nabla_\xi)\tr(\alpha_jW) - \nabla_x A_0 \cdot \nabla_{\xi} \tr(W)=0,
\end{equation*}
in view of \eqref{RegPW-WP}. Now we apply the lemma to $\rho=\tr(W)(t,x,\xi) \in L^{\infty}_tL^2_{x,\xi}$ with $z=(x,\xi)\in\mathbb R^6$ and
\begin{equation*}
J(t,x,\xi):= \begin{pmatrix}\tr(\alpha_1W)\\ \tr(\alpha_2W)\\\tr(\alpha_3W)\\ \partial_{x_1}A \cdot \tr(\alpha W)\\ 
\partial_{x_2}A \cdot \tr(\alpha W)\\ \partial_{x_3}A \cdot \tr(\alpha W) \\  -\partial_{x_1}A_0  \tr( W)\\ 
-\partial_{x_2}A_0 \tr( W)\\ -\partial_{x_3}A_0  \tr( W)  \end{pmatrix}.
\end{equation*}
Then $J \in L^1_{x,\xi}$ since $\nabla A, \nabla A_0 \in L^2_x$ and $W\in L^2_{x,\xi}$. We conclude that
\begin{equation*}
\iint_{\mathbb R^3\times\mathbb R^3}\tr(W)(t,x,\xi)dxd\xi=\iint_{\mathbb R^3\times\mathbb R^3}\left(\tr(W_+^{\mathrm{in}})+\tr(W_-^{\mathrm{in}})\right)(x,\xi)dxd\xi=1,
\end{equation*}
for a.e. $t\ge 0$.

\smallskip
Let us explain this equality in terms of the density operator $
R^{\varepsilon}(t)\in \mathfrak{S}^1(\mathfrak{H})$
that is the solution of the von Neumann equation \eqref{eq:Dirac equation} associated to the Dirac operator. For the integral kernel of $R^{\varepsilon}(t)$ we have
\begin{equation*}
R^{\varepsilon}(t,X,Y)=\sum_{j=1}^{\infty}\lambda^\varepsilon_j\psi^\varepsilon_j(t,X)\overline{\psi^\varepsilon_j(t,Y)}^\top,
\end{equation*}
with
\begin{equation*}
\lambda^\varepsilon_j\ge 0,\quad\sum_{j=1}^{\infty}\lambda^\varepsilon_j=1,\qquad\int_{\mathbb R^3}\langle \psi^\varepsilon_j(t,X),{\psi^\varepsilon_k(t,X)}\rangle_{\mathbb{C}^4} dX=\delta_{jk}.
\end{equation*}
Assume that
\begin{equation*}
\sum_{j=1}^{\infty}\lambda^\varepsilon_j\psi^\varepsilon_j(t,\cdot) \overline{\psi^\varepsilon_j(t,\cdot)}^\top\to\nu(t,\cdot)\in\mathcal M(\mathbb R^3),
\end{equation*}
vaguely as $\varepsilon\to 0$. Then, for each $\chi\in C_c(\mathbb R^3)$ and each $\phi\in C(\mathbb R^3)$ s.t.
\begin{equation*}
\mathbf 1_{B(0,1)}\le\phi\le\mathbf 1_{B(0,2)}\quad\text{ on }\mathbb R^3,
\end{equation*}
one has
\begin{equation*}
\begin{aligned}
\int_{\mathbb R^6_{x,\xi}}\phi\left(\frac\xi{R}\right)|\chi(x)|^2\tr(W)(t,x,\xi)dxd\xi
\\
=\lim_{\varepsilon\to 0}\sum_{j=1}^{\infty}\lambda^\varepsilon_j\int_{\mathbb R^6_{x,\xi}}\tr(W^{\varepsilon}[\chi\psi^\varepsilon_j(t,\cdot)])(x,\xi)\phi\left(\frac\xi{R}\right)dxd\xi
\\
=\lim_{\varepsilon\to 0}\sum_{j=1}^{\infty}\lambda^\varepsilon_j\left\langle \phi\left(\frac{\varepsilon D_x}{R}\right)\chi\psi^\varepsilon_j(t,\cdot),\chi\psi^\varepsilon_j(t,\cdot)\right\rangle_{\mathfrak{H}}&.
\end{aligned}
\end{equation*}
This can be recast as
\begin{equation*}
\begin{aligned}
0\le\int_{\mathbb R^3}|\chi(x)|^2\tr(\nu)(t,dx)-\int_{\mathbb R^6_{x,\xi}}\phi\left(\frac\xi{R}\right)|\chi(x)|^2\tr(W)(t,x,\xi)dxd\xi
\\
=\lim_{\varepsilon\to 0}\sum_{j=1}^{\infty}\lambda^\varepsilon_j\left\langle\left(I-\phi\left(\frac{\varepsilon D_x}{R}\right)\right)\chi\psi^\varepsilon_j(t,\cdot),\chi\psi^\varepsilon_j(t,\cdot)\right\rangle_{\mathfrak{H}}
\\
=\lim_{\varepsilon\to 0}\sum_{j=1}^{\infty}\lambda^\varepsilon_j\int_{\mathbb R^3}\left(1-\phi\left(\frac{\xi}R\right)\right)\left|\widehat{\chi\psi^\varepsilon_j}(t,\xi)\right|^2\frac{d\xi}{(2\pi)^3}.
\end{aligned}
\end{equation*}
Therefore
\begin{equation*}
0\le\int_{\mathbb R^3}|\chi(x)|^2\tr(\nu)(t,dx)-\int_{\mathbb R^6_{x,\xi}}\phi\left(\frac\xi{R}\right)|\chi(x)|^2\tr(W)(t,x,\xi)dxd\xi\to 0,
\end{equation*}
as $R\to\infty$, meaning that
\begin{equation*}
\tr(\nu)(t,\cdot)=\int_{\mathbb R^3_\xi}\tr(W)(t,\cdot,\xi)d\xi=:\rho(t,\cdot),
\end{equation*}
if and only if
\begin{equation*}
\sup_{0< \varepsilon \leq 1}\sum_{j=1}^{\infty}\lambda^\varepsilon_j\int_{|\xi|>R/\varepsilon}\left|\widehat{\chi\psi^\varepsilon_j(t,\xi)}\right|^2d\xi \longrightarrow 0 \quad \text{ as } R\rightarrow \infty.
\end{equation*}
Besides
\begin{equation*}
\iint_{\mathbb R^6_{x,\xi}}\tr(W)(t,x,\xi)dx d\xi=\int_{\mathbb R^3}\nu(t,dx)=1=\sum_{j=1}^{\infty}\lambda_j^\varepsilon=\sum_{j=1}^{\infty}\lambda_j^\varepsilon\int_{\mathbb R^3}|\psi_j^\varepsilon(t,x)|^2dx,
\end{equation*}
if and only if
\begin{equation*}
\sum_{j=1}^{\infty}\lambda^\varepsilon_j|\psi^\varepsilon_j(t,\cdot)|^2\to\tr(\nu(t,\cdot)),
\end{equation*}
narrowly as $\varepsilon\to 0$, meaning in particular that the family of functions
\begin{equation*}
\mathbb R^3\ni x\mapsto\sum_{j=1}^{\infty}\lambda^\varepsilon_j|\psi^\varepsilon_j(t,x)|^2\in[0,+\infty),
\end{equation*}
is tight on $\mathbb R^3$. 

For the current density we assume that
\begin{equation*}
\sum_{j=1}^{\infty}\lambda^\varepsilon_j\psi^\varepsilon_j(t,\cdot) \overline{\alpha_k \psi^\varepsilon_j(t,\cdot)}^\top\to\mu_k(t,\cdot)\in\mathcal M(\mathbb R^3),
\end{equation*}
vaguely for $k=1,2,3$. Taking $\chi$ as above we obtain
\begin{equation*}
\begin{aligned}
\int_{\mathbb R^6_{x,\xi}}\phi\left(\frac\xi{R}\right)|\chi(x)|^2\tr(\alpha_k W)(t,x,\xi)dxd\xi
\\
=\lim_{\varepsilon\to 0}\sum_{j=1}^{\infty}\lambda^\varepsilon_j\int_{\mathbb R^6_{x,\xi}}\tr(W^{\varepsilon}[\chi\psi^\varepsilon_j(t,\cdot),\chi\alpha_k \psi^\varepsilon_j(t,\cdot)])(x,\xi)\phi\left(\frac\xi{R}\right)dxd\xi
\\
=\lim_{\varepsilon\to 0}\sum_{j=1}^{\infty}\lambda^\varepsilon_j\left\langle \phi\left(\frac{\varepsilon D_x}{R}\right)\chi\psi^\varepsilon_j(t,\cdot),\chi\alpha_k\psi^\varepsilon_j(t,\cdot)\right\rangle_{\mathfrak{H}}&,
\end{aligned}
\end{equation*}
where in the second line $W^{\varepsilon}[\chi \psi_j^{\varepsilon},\chi \alpha_k \psi_j^{\varepsilon}]$ is defined as in \eqref{eq:wigner transform general} with $f = \chi \psi_j^{\varepsilon}$ and $g = \chi \alpha_k \psi_j^{\varepsilon}$. Again, we rewrite this as
\begin{equation*}
\begin{aligned}
0\le\int_{\mathbb R^3}|\chi(x)|^2\tr(\mu_k)(t,dx)-\int_{\mathbb R^6_{x,\xi}}\phi\left(\frac\xi{R}\right)|\chi(x)|^2\tr(\alpha_k W)(t,x,\xi)dxd\xi
\\
=\lim_{\varepsilon\to 0}\sum_{j=1}^{\infty}\lambda^\varepsilon_j\left\langle\left(I-\phi\left(\frac{\varepsilon D_x}{R}\right)\right)\chi\psi^\varepsilon_j(t,\cdot),\chi\alpha_k\psi^\varepsilon_j(t,\cdot)\right\rangle_{\mathfrak{H}}
\\
=\lim_{\varepsilon\to 0}\sum_{j=1}^{\infty}\lambda^\varepsilon_j\int_{\mathbb R^3}\left(1-\phi\left(\frac{\xi}R\right)\right)\langle\widehat{\chi\psi^\varepsilon_j}(t,\xi),\widehat{\chi\alpha_k \psi^\varepsilon_j}(t,\xi)\rangle_{\mathbb{C}^4}\frac{d\xi}{(2\pi)^3}.
\end{aligned}
\end{equation*}
Next we make use of the Cauchy-Schwarz inequality to bound
\begin{align*}
    |\langle\widehat{\chi\psi^\varepsilon_j}(t,\xi),\widehat{\chi\alpha_k \psi^\varepsilon_j}(t,\xi)\rangle_{\mathbb{C}^4}| &\leq |\widehat{\chi\psi^\varepsilon_j}(t,\cdot)||\widehat{\chi\alpha_k\psi^\varepsilon_j}(t,\cdot)|\\
    &= |\widehat{\chi\psi^\varepsilon_j}(t,\cdot)|^2.
\end{align*}
Therefore, we can reduce the case of $J_k^{\varepsilon}$ to the case of $\rho^{\varepsilon}$, which we have already proven, and we obtain
\begin{equation*}
\tr(\mu_k)(t,\cdot)=\int_{\mathbb R^3_\xi}\tr(\alpha_k W)(t,\cdot,\xi)d\xi=:J_k(t,\cdot).
\end{equation*}
This completes the proof.
\end{proof}
\end{lemma}

\begin{proof}[Proof of Theorem \ref{thm:main1}]
Due to condition \eqref{eq:assumption lambda} and the considerations in Remark \ref{remark L^2}, we obtain
\begin{equation*}
    \|W^{\varepsilon}_{\mathrm{in}}\|_{L^2(\mathbb{R}^3_x \times \mathbb{R}^3_{\xi})} \leq C,
\end{equation*}
where $C$ is independent of $t$. The Dirac-Wigner equation conserves the $L^2$ norm due to Corollary \ref{thm:existence wigner}. Thus we conclude that $W^{\varepsilon} \in L^{\infty}(\mathbb{R}_t, L^2(\mathbb{R}^3_x \times \mathbb{R}^3_{\xi}))$ and thus Proposition \ref{proposition general} applies. Theorem \ref{thm:main1} then follows from the lemmas after Proposition \ref{proposition general}.
\end{proof}
\subsection{Proof of Theorem \ref{thm:main2}}
\label{sec:proofmain2}
We repeat the projected limit of the Wigner equation, i.e. \eqref{eq:Wigner matrix equation limit projected},
\begin{align}
     \Pi_\pm \partial_t W \Pi_\pm &= \frac{1}{2} \Pi_\pm\left(\left\{ P ,W\right\} -\left\{W, P\right\} \right)\Pi_\pm, \quad W|_{t=0} = W^{\mathrm{in}}_{\pm} := \Pi_{\pm}W_{\mathrm{in}}\Pi_{\pm}, \\
     [P,W] &= 0, \quad \text{for all } t.
\end{align}
Recall that
\begin{equation}
     P(t,x,\xi)=\alpha \cdot (\xi-A(t,x))+\beta-A_0(t,x),
\end{equation}
and
\begin{equation}
    P = \lambda_+ \Pi_+ + \lambda_- \Pi_-, \qquad \lambda_{\pm} = \pm \sqrt{1+|\xi-A|^2} -A_0.
\end{equation}
Moreover, since $[P,W]=0$ we have
\begin{equation}
    [\Pi_{\pm},W] = 0.
\end{equation}
We state the following two lemmas, which will be needed further down the line.
\begin{lemma}
\label{lemma derivation Pi}
    Let $\Pi$ be a projection and let $D$ be a derivation. Then
    \begin{equation}
        \Pi (D\Pi) \Pi = 0.
        \label{eq:identity Pi}
    \end{equation}
    If $\Pi_1 + \Pi_2 = \mathrm{Id}$ are projections, then in addition
    \begin{equation}
        \Pi_1(D\Pi_{2}) \Pi_1 = \Pi_2(D\Pi_{1}) \Pi_2 = 0.
        \label{eq:identity Pi 2}
    \end{equation}
    \begin{proof}
        Since $\Pi^2 = \Pi$,
        \begin{equation*}
    D \Pi = D \Pi^2 =  \Pi D \Pi + (D \Pi)\Pi.
\end{equation*}
Multiplying by $\Pi$ from the left yields the first claim. For the second claim, use $D\Pi_1 = - D\Pi_2$.
    \end{proof}
\end{lemma}
\begin{lemma}
\label{eq:lemma derivation projection}
    Let $D$ be a derivation and let $\Pi$ be a projection and $W$ be a matrix that commutes with $\Pi$, i.e. $[W,\Pi]=0$. Then
    \begin{equation}
        \Pi (DW) \Pi = D(\Pi W \Pi) - [\Pi W,[\Pi,D\Pi]] .
    \end{equation}
    \begin{proof}
        By the product rule for derivations,
        \begin{equation*}
            D(\Pi W \Pi) = \Pi (DW) \Pi + (D\Pi) W\Pi + \Pi W(D\Pi).
        \end{equation*}
        Since $[W,\Pi]=0$ and thus $W\Pi=\Pi W\Pi$ as well as $\Pi(D\Pi)\Pi =0$ (Lemma \ref{lemma derivation Pi}),
        \begin{align*}
            (D\Pi)W\Pi &= (D\Pi)\Pi W\Pi- W\Pi(D\Pi)\Pi \\
            &= -[W\Pi,(D\Pi)\Pi].
        \end{align*}
        Similary,
        \begin{align*}
            \Pi W(D\Pi) &= \Pi W\Pi(D\Pi)- \Pi(D\Pi)\Pi W \\
                &= [W\Pi,\Pi(D\Pi)].
        \end{align*}
        Thus,
        \begin{equation*}
            (D\Pi) W\Pi + \Pi W(D\Pi) = [W\Pi,[\Pi,(D\Pi)].
        \end{equation*}
        This proves the claim.
    \end{proof}
\end{lemma}
Denote 
\begin{equation}
    W_j = \Pi_j W \Pi_j,
\end{equation}
where $j=\pm$.
Then, using $\Pi_{\pm}\Pi_{\mp} = 0$, $\Pi_{\pm}^2 = \Pi_{\pm}$ and the product rule for the Poisson bracket,
\begin{align}
    \Pi_j \partial_t W \Pi_j &= \frac{1}{2} \Pi_j \big( \{ ( \lambda_+ \Pi_+ +\lambda_-\Pi_-),W\} - \{W,( \lambda_+ \Pi_+ +\lambda_-\Pi_- ) \} \big) \Pi_j \nonumber \\
    &= \Pi_j \{\lambda_j,W\}\Pi_j \nonumber \\
    &\qquad +\frac{1}{2} 
    \lambda_+ \Pi_j ( \{\Pi_+,W\}-\{W,\Pi_+\})\Pi_j \nonumber\\
    &\qquad +\frac{1}{2} 
    \lambda_- \Pi_j ( \{\Pi_-,W\}-\{W,\Pi_-\})\Pi_j \nonumber \\
    &= (a) + (b),
\end{align}
where
\begin{equation}
    (a) := \Pi_j \{\lambda_j,W\}\Pi_j,
\end{equation}
and
\begin{equation}
    (b) := \frac{1}{2} 
    \lambda_+ \Pi_j ( \{\Pi_+,W\}-\{W,\Pi_+\})\Pi_j  +\frac{1}{2} 
    \lambda_- \Pi_j ( \{\Pi_-,W\}-\{W,\Pi_-\})\Pi_j.
\end{equation}
Consider first $(a)$. 

Using Lemma \ref{eq:lemma derivation projection} and the fact that the Poisson bracket is a derivation,
\begin{align}
    \Pi_j \{\lambda_j,W\}\Pi_j &= \{\lambda_j,W_j\} - [W_j,[\Pi_j,\{\lambda_j,\Pi_j\}]
    \label{eq: (a) term}.
\end{align}
Now consider $(b)$. Note that using $\Pi_+ =I-\Pi_-$ we have
\begin{align*}
    \Pi_+ ( \{\Pi_+,W\}-\{W,\Pi_+\})\Pi_+ = - \Pi_+ ( \{\Pi_-,W\}-\{W,\Pi_-\})\Pi_+,
\end{align*}
and vice versa. Thus we only have to consider a term of the form
\begin{equation}
    \Pi_j ( \{\Pi_k,W\}-\{W,\Pi_k\})\Pi_j,
\end{equation}
for $j\neq k$. We have the identity 
\begin{equation}
    A\{B,C\} - \{A,B\}C = \{AB,C\} - \{A,BC\}.
    \label{eq:Poisson identity}
\end{equation}
Choosing $A=\Pi_j$, $B=\Pi_k$ and $C=W$ yields, using $\Pi_j\Pi_k=0$
\begin{align*}
    \Pi_j\{\Pi_k,W\} - \{\Pi_j,\Pi_k\}W &= \{\Pi_j\Pi_k,W\}-\{\Pi_j,\Pi_k W\} \\ &= -\{\Pi_j,\Pi_k W\},
\end{align*}
and choosing $A=W$, $B=\Pi_k$ and $C=\Pi_j$,
\begin{align*}
    W \{\Pi_k,\Pi_j\} - \{W,\Pi_k\}\Pi_j &= \{W\Pi_k,\Pi_j\} - \{W,\Pi_k\Pi_j\} \\ &= \{W\Pi_k,\Pi_j\}.
\end{align*}
Therefore,
\begin{align*}
    \Pi_j ( \{\Pi_k,W\}-\{W,\Pi_k\})\Pi_j &= \{\Pi_j,\Pi_k\} W \Pi_j - \{\Pi_j,\Pi_k W\}\Pi_j \\
    &\qquad + \Pi_j\{W \Pi_k,\Pi_j\} - \Pi_j W\{\Pi_k,\Pi_j\}.
\end{align*}
Using again \eqref{eq:Poisson identity} with $A= \Pi_j$, $B=\Pi_kW =W\Pi_k$ and $C= \Pi_j$, as well as $\Pi_j\Pi_k=0$, we observe that the two terms in the middle vanish. Thus,
\begin{align*}
    \Pi_j ( \{\Pi_k,W\}-\{W,\Pi_k\})\Pi_j &= \{\Pi_j,\Pi_k\} W \Pi_j   - \Pi_j W\{\Pi_k,\Pi_j\}.
\end{align*}
It remains to observe that we can rewrite $\{\Pi_j,\Pi_k\}$ and $\{\Pi_k,\Pi_j\}$ as
\begin{align*}
    \{\Pi_j,\Pi_k\} &= - \Pi_j\{\Pi_k,\Pi_k\} +\{\Pi_j,\Pi_k\}\Pi_k \\
    \{\Pi_k,\Pi_j\} &= - \{\Pi_k,\Pi_k\} \Pi_j + \Pi_k\{\Pi_k,\Pi_j\},
\end{align*}
which again is an application of \eqref{eq:Poisson identity}. This allows us to write
\begin{align*}
    \Pi_j ( \{\Pi_k,W\}-\{W,\Pi_k\})\Pi_j &= [W_j,\Pi_j\{\Pi_k,\Pi_k\}\Pi_j].
\end{align*}
Then $(b)$ becomes
\begin{align*}
    (b) &= \frac{1}{2} 
    \lambda_+ \Pi_j ( \{\Pi_+,W\}-\{W,\Pi_+\})\Pi_j  +\frac{1}{2} 
    \lambda_- \Pi_j ( \{\Pi_-,W\}-\{W,\Pi_-\})\Pi_j \\ &= \mp\frac{1}{2}(\lambda_+-\lambda_-)[W_\pm, \Pi_{\pm} \{\Pi_\mp,\Pi_\mp\}\Pi_\pm] \\
    &= \mp\langle \xi-A\rangle[W_\pm, \Pi_{\pm} \{\Pi_\mp,\Pi_\mp\}\Pi_\pm].
\end{align*}
where we used $\lambda_+-\lambda_- =\langle \xi-A\rangle -A_0 - (-\langle \xi-A\rangle -A_0) = 2 \langle \xi-A\rangle $.

Collecting the results for $(a)$ and $(b)$ yields
\begin{align*}
    \Pi_j \partial_t W \Pi_j 
     &= \{\lambda_j,W_j\} - [W_j,[\Pi_j,\{\lambda,\Pi_j\}] \\
    &\qquad \mp \langle \xi-A\rangle [W_\pm,\Pi_\pm\{\Pi_\mp,\Pi_\mp\}\Pi_\pm].
\end{align*}
It remains to extract $\partial_t W_j$ from $\Pi_j \partial_t W \Pi_j$. But this is simply another application of Lemma \ref{eq:lemma derivation projection}, which yields
\begin{equation}
    \Pi_j  \partial_t W \Pi_j = \partial_t W_j - [W_j,[\Pi_j,\partial_t\Pi_j]].
\end{equation}
Using the definitions \eqref{eq:berry term} and \eqref{eq:poissonian curvature} for the Berry term $H_{\mathrm{be}}$ and the Poissoninan curvature $H_{\mathrm{pc}}$ yields
\begin{align}
    \partial_t W_j &= \{\lambda_j,W_j\} - [W_j,[\Pi_j,\{\lambda_j,\Pi_j\}-\partial_t\Pi_j] ] \nonumber\\ &\qquad \mp \langle \xi-A\rangle [W_\pm,\Pi_\pm\{\Pi_\mp,\Pi_\mp\}\Pi_\pm] \nonumber\\
    &= \{\lambda_j,W_j\} + [H_{\text{be}},W_j] + [H_{\text{pc}},W_j]. 
\end{align}
This completes the proof for \eqref{eq:relativistic Vlasov}.

Taking the trace in $\mathbb{C}^4$ in \eqref{eq:relativistic Vlasov} and denoting $f_{\pm} =\tr_{\mathbb{C}^4} W_\pm$ yields
\begin{equation}
    \partial_t f_{\pm} = \{\lambda_\pm,f_{\pm}\},\label{eq:relativistic Vlasov trace}
\end{equation}
The Poisson bracket $\{\lambda_{\pm},f_{\pm}\}$ is easily calculated:
\begin{align}
    \{\lambda_{\pm},f_{\pm}\} &= \nabla_x \lambda_{\pm} \cdot \nabla_{\xi} f_{\pm} - \nabla_{\xi}\lambda_{\pm} \cdot \nabla_x f_{\pm}\nonumber \\
    &= \pm\left(\frac{-\nabla_x A \cdot (\xi-A)}{\langle \xi-A\rangle} \cdot \nabla_{\xi} f_{\pm} - \frac{\xi-A}{\langle \xi-A\rangle} \cdot \nabla_x f_{\pm}\right) -\nabla A_0\cdot \nabla_{\xi} f_{\pm}.
\end{align}
The change of variables $v=\xi-A$ and $B=\nabla \times A$, $E=\nabla A_0 - \partial_t A$ yields the relativistic transport equation for the phase space density $f_+$ for electrons and $f_-$ for positrons respectively,
\begin{equation}
    f_{\pm} \pm \frac{v}{\langle v \rangle}\cdot \nabla_x f_{\pm} + (E\pm \frac{v}{\langle v \rangle} \times B) \cdot \nabla_v f_{\pm} = 0.
\end{equation}
For the densities $\rho^{\varepsilon}$ and $J^{\varepsilon}$ (i.e. zeroth order moments of $W^{\varepsilon}$ we have according to Theorem \ref{thm:main1}, 
\begin{align*}
    \rho^{\varepsilon}(t,x) \rightarrow \rho(t,x) = \int \tr_{\mathbb{C}^4}(W(x,\xi)) d\xi &= \int \tr_{\mathbb{C}^4}((\Pi_+ +\Pi_-)W)(t,x,\xi) d\xi \\
    &= \int \tr_{\mathbb{C}^4}(\Pi_+W)(t,x,\xi)d \xi +\int\tr_{\mathbb{C}^4}(\Pi_-W)(t,x,\xi) d\xi \\
    &= \int (f_++f_- )(t,x,v) dv
\end{align*}
for almost every $t$ narrowly on $\mathbb{R}^3$. Similarly, we can extract the relativistic velocity $\tfrac{v}{\langle v \rangle}$ with $v=\xi-A$ from $J^{\varepsilon}$, i.e.
\begin{align*}
    J^{\varepsilon}(t,x) \rightarrow J(t,x) &= \int \tr_{\mathbb{C}^4}(\alpha_k(\Pi_+ +\Pi_-)W(\Pi_+ +\Pi_-))(t,x,\xi) d\xi \\
    &= \int \tr_{\mathbb{C}^4}(\alpha_k(\Pi_+W\Pi_+)) (t,x,\xi)+\tr_{\mathbb{C}^4}(\alpha_k(\Pi_-W\Pi_-)) (t,x,\xi)d\xi \\
   &= \int \frac{v}{\langle v \rangle}\tr_{\mathbb{C}^4}( W\Pi_+)(t,x,\xi) - \frac{v}{\langle v \rangle}\tr_{\mathbb{C}^4}( W\Pi_-)(t,x,\xi) d\xi \\
   &= \int \frac{v}{\langle v \rangle} f_+(t,x,v) - \frac{v}{\langle v \rangle}f_-(t,x,v) dv
\end{align*}
for almost every $t$ narrowly on $\mathbb{R}^3$, where in the last step we used \eqref{eq: identity Pi alpha} and the cyclicity of the trace. This concludes the proof of Theorem \ref{thm:main2}.

\appendix

\section{Derivation of the Dirac-Wigner equation}
\label{subsection:Derivation of the Dirac-Wigner equation}

In this section we derive the Dirac-Wigner equation \eqref{eq:Wigner matrix equation_2} for the Wigner transform $W^{\varepsilon}$. To this end, we will write the Dirac-von Neumann equation \eqref{eq:Dirac equation} in the center of mass and relative variables \eqref{eq: weyl variables} and then take the Fourier transform in $y$. In order to arrive at an equation of the form \eqref{eq:Wigner matrix equation_2} we have to pay some attention to how to regroup the terms, which is not immediately obvious.

Let $R$ be the solution of the Dirac-von Neumann equation \eqref{eq:Dirac equation}, i.e. 
\begin{align*}
    i\varepsilon \partial_t R(X,Y) =& [\alpha\cdot(-i\varepsilon\nabla-A(\cdot))+\beta-A_0(\cdot),R](X,Y) \\
= &-i\varepsilon \partial_{X_k} \alpha_k R(X,Y) - i\varepsilon\partial_{Y_k} R(X,Y)\alpha_k \\ &- \left(A_k(X)\alpha_k R(X,Y)-A_k(Y) R(X,Y)\alpha_k\right) +[\beta,R(X,Y)] \\ &- \left(A_0(X) -A_0(Y) \right)R(X,Y).
\end{align*}
Now introduce the variables
\begin{align}
    x=\frac{X+Y}{2}, && y=\frac{X-Y}{\varepsilon},\label{eq: weyl variables}
\end{align}
such that
\begin{align}
    \partial_{X_k} = \frac{1}{2}\partial_{x_k} + \frac{1}{\varepsilon}\partial_{y_k}, && \partial_{Y_k} = \frac{1}{2}\partial_{x_k} - \frac{1}{\varepsilon}\partial_{y_k}.
\end{align} Denote $$\tilde{R}(x,y) = R(x+\frac{\varepsilon y}{2}, x-\frac{\varepsilon y}{2}).$$
Then 
\begin{align*}
    i\varepsilon \partial_t \tilde{R}(x,y) = &-i\varepsilon \left(\frac12 \partial_{x_k} + \frac{1}{\varepsilon}\partial_{y_k}\right)\alpha_k \tilde{R}(x,y) - i\varepsilon\left(\frac12 \partial_{x_k} - \frac{1}{\varepsilon}\partial_{y_k}\right)\tilde{R}(x,y)\alpha_k \\ &- \left(A_k(x+\frac{\varepsilon y}{2})\alpha_k \tilde{R}(x,y)-A_k(x-\frac{\varepsilon y}{2}) \tilde{R}(x,y)\alpha_k\right) +[\beta,\tilde{R}(x,y)] \\
    &-\left(A_0(x+\frac{\varepsilon y}{2})-A_0(x-\frac{\varepsilon y}{2})\right) \tilde{R}(x,y).
\end{align*}
Dividing by $i\varepsilon$ yields
\begin{align*}
    \partial_t \tilde{R}(x,y) = &\frac{1}{i\varepsilon}\left((-i\partial_{y_k})[\alpha_k,\tilde{R}(x,y)] + [\beta,\tilde{R}(x,y)]\right)  -\frac{1}{2}\partial_{x_k}[\alpha_k, \tilde{R}(x,y)]_+ \\ &\qquad - \frac{1}{i\varepsilon}\left(A_k(x+\frac{\varepsilon y}{2})\alpha_k \tilde{R}(x,y)-A_k(x-\frac{\varepsilon y}{2}) \tilde{R}(x,y)\alpha_k\right) \\ &\qquad -\frac{1}{i\varepsilon} \left(A_0(x+\frac{\varepsilon y}{2})-A_0(x-\frac{\varepsilon y}{2})\right) \tilde{R}(x,y). 
\end{align*}
The trick is now to regroup the terms involving $A_k(x\pm \tfrac{\varepsilon y}{2})$ in the following way:
\begin{align*}
& \frac{1}{i\varepsilon}\left(A_k(x+\frac{\varepsilon y}{2})\alpha_k \tilde{R}(x,y)-A_k(x-\frac{\varepsilon y}{2}) \tilde{R}(x,y)\alpha_k\right) \\
    &\qquad =  \frac{1}{i\varepsilon}\left(\frac{A_k(x+\frac{\varepsilon y}{2})+A_k(x-\frac{\varepsilon y}{2})}{2}\right)[\alpha_k,\tilde{R}(x,y)] \\ &\qquad \quad +\frac{1}{2}\left(\frac{A_k(x+\frac{\varepsilon y}{2})-A_k(x-\frac{\varepsilon y}{2})}{i\varepsilon}\right)[\alpha_k, \tilde{R}(x,y)]_+.
\end{align*}
Therefore,
\begin{align*}
    \partial_t \tilde{R}(x,y) = &\frac{1}{i\varepsilon}\left((-i\partial_{y_k})[\alpha_k,\tilde{R}(x,y)] + [\beta,\tilde{R}(x,y)]\right)   -\frac{1}{2}\partial_{x_k}[\alpha_k, \tilde{R}(x,y)]_+ \\ &\qquad -  \frac{1}{i\varepsilon}\left(\frac{A_k(x+\frac{\varepsilon y}{2})+A_k(x-\frac{\varepsilon y}{2})}{2}\right)[\alpha_k,\tilde{R}(x,y)] \\ &\qquad -\frac{1}{2}\left(\frac{A_k(x+\frac{\varepsilon y}{2})-A_k(x-\frac{\varepsilon y}{2})}{i\varepsilon}\right)[\alpha_k, \tilde{R}(x,y)]_+\\ &\qquad -\left(\frac{A_0(x+\frac{\varepsilon y}{2})-A_0(x-\frac{\varepsilon y}{2})}{i\varepsilon}\right) \tilde{R}(x,y).
\end{align*}
Taking the Fourier transform in $y$ yields the following preliminary form of the Wigner equation,
\begin{align}
    \partial_t W^{\varepsilon} = &\frac{1}{i\varepsilon}[\alpha \cdot \xi + \beta,W^{\varepsilon}]  +\frac{1}{2} \left( \left\{ \alpha\cdot \xi + \beta, W^{\varepsilon}\right\}-\left\{ W^{\varepsilon},\alpha\cdot \xi + \beta\right\} \right) \nonumber \\ &- \frac{1}{i\varepsilon}\tau[A_k][\alpha_k,W^{\varepsilon}] -\frac{1}{2}\theta[A_k][\alpha_k,W^{\varepsilon}] - \theta[A_0] W^{\varepsilon}.
    \label{eq:preliminary wigner dirac}
\end{align}
where $\tau$ and $\theta$ are the pseudo-differential operators defined as
\begin{equation}
    (\theta[g]f)(x,\xi) := \int_{\mathbb{R}^6} \frac{1}{i\varepsilon}(g(x+\frac{\varepsilon y}{2})-g(x-\frac{\varepsilon y}{2}))f(x,\eta) e^{-i (\xi-\eta)\cdot y} d \eta d y, 
\end{equation}
cf. \eqref{eq:PDO} and
\begin{equation}
    (\tau[g]f)(x,\xi)  := \int_{\mathbb{R}^6} \frac{1}{2}(g(x+\frac{\varepsilon y}{2})+g(x-\frac{\varepsilon y}{2}))f(x,\eta) e^{-i (\xi-\eta)\cdot y} d \eta d y. 
\end{equation}
In order to transform equation \eqref{eq:preliminary wigner dirac} into an equation of the form \eqref{eq:Wigner matrix equation_2} we add and subtract $-\tfrac{1}{i\varepsilon} [\alpha_k A_k(x),W^{\varepsilon}]$ as well as $\tfrac{1}{2} \nabla_x A_k \cdot \nabla_{\xi} (\alpha_k W^{\varepsilon} + W^{\varepsilon}\alpha_k)$ and $\nabla A_0  \cdot \nabla_{\xi} W^{\varepsilon}$. Moreover, we define the second order finite difference pseudodifferential operator operator
\begin{align}
    (\Delta[g]f)(x,\xi) &= \frac{1}{i\varepsilon^2}((\tau[g]-g)f)(x,\xi)  \\ &= \frac{1}{2i}\int_{\mathbb{R}^6} \frac{g(x+\frac{\varepsilon y}{2})-2g(x)+g(x-\frac{\varepsilon y}{2})}{\varepsilon^2}f(x,\eta) e^{-i(\xi-\eta)\cdot y} d\eta d y,
\end{align}
cf. \eqref{eq:PDO_2}. Also observe that $[A_0,W^{\varepsilon}]= 0$ since $A_0$ is a scalar. This yields the {Dirac-Wigner equation} \eqref{eq:Wigner matrix equation_2}, i.e.
\begin{align}
\partial_t W^{\varepsilon} &= \frac{1}{i\varepsilon}[P,W^{\varepsilon}]  +\frac{1}{2} \left( \left\{ P, W^{\varepsilon}\right\}-\left\{ W^{\varepsilon},P\right\} \right)  +  r^{\varepsilon},\\
P &= \alpha \cdot (\xi-A) + \beta-A_0,
\end{align}
with the remainder $r^{\varepsilon}$ given by
\begin{align}
r^{\varepsilon} 
&= \varepsilon \Delta[A_k][\alpha_k,W^{\varepsilon}] +\frac{1}{2}(-\nabla_x A_k\cdot \nabla_{\xi}-\theta[A_k])[\alpha_k, W^{\varepsilon}]_+ 
 -(-\nabla A_0 \cdot \nabla_{\xi} -\theta [A_0])W^{\varepsilon}. 
\end{align}

\section{Explicit calculations}
In the following we set
\begin{align}
    v(t,x,\xi) &= \xi-A(t,x), & P_0(t,x,\xi) &= \alpha\cdot v(t,x,\xi)+\beta, \\
    \lambda_\pm(t,x,\xi) &= \pm \langle v(t,x,\xi) \rangle -A_0(t,x), & \Pi_\pm(t,x,\xi) &= \frac{1}{2}\left( I_4  \pm \frac{P_0(t,x,\xi)}{\langle v(t,x,\xi) \rangle}\right),
\end{align}
and
\begin{equation*}
{\varepsilon}_{km\ell}B_{\ell} = \partial_k A_m-\partial_m A_k.
\end{equation*}

\subsection{Dirac matrices}
\label{sec: identities gamma}
Denote
\begin{equation}
    \gamma^5 = -i \alpha_1\alpha_2\alpha_3.
\end{equation}
Then we can write the product of two Dirac matrices as
\begin{equation}
    \alpha_j\alpha_k = \delta_{jk}I_4+ i \gamma^5 \varepsilon_{jkm}\alpha_m.
    \label{eq:product Dirac matrices}
\end{equation}
Then we have
\begin{align}
    [\alpha_k,\alpha_m] &= 2(\delta_{km}-\alpha_m\alpha_k) = 2i\gamma^5 \varepsilon_{kmj}\alpha_j,
    \label{eq: identity commutator alpha alpha} \\
    [\beta,\alpha_m] &= 2\beta \alpha_m.
    \label{eq: identity commutator beta alpha} 
\end{align}
Using \eqref{eq: identity commutator alpha alpha} we deduce that
\begin{align}
    [\alpha \cdot a, \alpha \cdot b] 
    &=2i\gamma^5 \alpha \cdot (a \times b).
    \label{eq: identity gamma 5}
\end{align}
Using \eqref{eq: identity commutator alpha alpha}, \eqref{eq: identity commutator beta alpha} and \eqref{eq: identity gamma 5} it follows that
\begin{align}
    [P_0,\alpha \cdot a] &= [\alpha \cdot v + \beta, \alpha \cdot a] \nonumber \\
    &= 2i\gamma^5 \alpha \cdot (v\times a) + 2\beta \alpha \cdot a.
    \label{eq: identity commutator P_0 alpha} 
\end{align}
Using $\alpha_m \alpha_k = -\alpha_k \alpha_m + 2\delta_{km}$ and $\alpha_j \beta + \beta \alpha_j = 0$ we have
\begin{align}
\Pi_\pm \alpha &= \alpha \Pi_\mp  \pm\frac{v}{\langle v \rangle},\label{eq: identity Pi alpha} \\
\alpha \Pi_\pm &= \Pi_\mp \alpha  \pm\frac{v}{\langle v \rangle}. \label{eq: identity alpha Pi}
\end{align}
\subsection{Derivatives and Poisson brackets}

We have the following identities.
\begin{align}
    \partial_{x_k} \langle v \rangle &= \frac{-v\cdot \partial_{x_k}A}{\langle v \rangle}, & \partial_{\xi_k} \langle v \rangle &= \frac{v_k}{\langle v \rangle}, \\
    \partial_{x_k} P_0 &= -\alpha \cdot \partial_{x_k} A, & \partial_{\xi_k} P_0 &= \alpha_k.
\end{align}
For $\lambda_\pm$ we have
\begin{align}
\partial_{x_k}\lambda_\pm &= \mp \frac{\partial_{x_k} A\cdot v }{\langle v \rangle} -\partial_{x_k}A_0,   \\
\partial_{\xi_k}\lambda_j &= \pm\frac{v_k}{\langle v \rangle}.
\end{align}
For $\Pi_\pm$ we have
\begin{align}\label{eq:derivative Pi x}
\partial_{x_k} \Pi_\pm &= \pm\frac{1}{2}\left(- \frac{\alpha \cdot \partial_{x_k}A}{\langle v \rangle} + \frac{P_0( v \cdot \partial_{x_k}A)}{\langle v \rangle^3}\right), \\
\partial_{\xi_k} \Pi_\pm &= \pm\frac{1}{2}\left(\frac{\alpha_k}{\langle v \rangle} - \frac{P_0 v_k}{\langle v \rangle^3}\right).\label{eq:derivative Pi xi}
\end{align}
Then we have for the Poisson brackets
\begin{equation}
    \{\lambda_\pm,\Pi_\pm\} = -\frac{1}{2\langle v \rangle}{\alpha \cdot (\frac{v}{\langle v \rangle}\times B)} \mp \frac{1}{2\langle v \rangle} \left( {\alpha \cdot \nabla A_0}- \frac{P_0 v \cdot \nabla A_0}{\langle v \rangle^2}\right).
    \label{eq: explicit formula Poisson bracket lambda Pi}
\end{equation}
and
\begin{align}
    \{\Pi_\pm,\Pi_\pm\} &= \frac{1}{4}\left\{\frac{P_0}{\langle v \rangle},\frac{P_0}{\langle v \rangle}\right\} \nonumber \\
    &= \frac{1}{4\langle v\rangle^2}\{P_0,P_0\}+\frac{1}{4\langle v\rangle}\left[\{ P_0,\frac{1}{\langle v \rangle} \},{P_0}\right]\nonumber \\ &=\frac{1}{4\langle v\rangle^2}\{\alpha \cdot v, \alpha \cdot v \} +\frac{1}{4\langle v\rangle}\left[\{ P_0,\frac{1}{\langle v \rangle} \},{P}\right] 
    \nonumber \\
    &= \frac{1}{4\langle v \rangle^2}\alpha \cdot (\alpha \times B) +\frac{1}{4\langle v\rangle}\left[\{ P_0,\frac{1}{\langle v \rangle} \},{P}\right]
    \nonumber. 
\end{align}
Now using $\alpha\cdot (\alpha \times B) = 2i\gamma^5 \alpha \cdot B$,
\begin{align}
    \{\Pi_\pm,\Pi_\pm\} &=  \frac{1}{2\langle v \rangle^2}\left(i\gamma ^5 \alpha  \cdot B +\frac{1}{4\langle v\rangle}\left[\{ P_0,\frac{1}{\langle v \rangle} \},{P}\right] 
    \right).
    \label{eq: explicit formula Poisson bracket Pi Pi}
\end{align}
For the time derivative of $\Pi_j$ we have
\begin{equation}
    \partial_t \Pi_\pm = \mp\frac{1}{2\langle v \rangle}\left(\alpha \cdot \partial_tA - \frac{P_0 v\cdot \partial_t A}{\langle v \rangle^2}\right).
    \label{eq: explicit formula time derivative Pi}
\end{equation}

\subsection{Berry term}
\label{sec:berry}
The Berry term \eqref{eq:berry term} is given by 
\begin{equation}
    H_{\text{be}} := [\Pi_\pm,\{\lambda_\pm, \Pi_\pm\}-\partial_t\Pi_\pm] 
\end{equation}
For the first term we have, using \eqref{eq: identity commutator P_0 alpha},\eqref{eq: explicit formula Poisson bracket lambda Pi}, $[I_4, M]=0$ for any matrix $M$ and $[P_0,P_0]=0$,
\begin{align*}
    \left[\Pi_\pm,\{\lambda_\pm, \Pi_\pm\}\right] &= \left[\pm\frac{1}{2}\frac{P_0}{\langle v \rangle}, -\frac{1}{2\langle v \rangle}{\alpha \cdot (\frac{v}{\langle v \rangle}\times B)} \mp \frac{1}{2\langle v \rangle} \left( {\alpha \cdot \nabla A_0}- \frac{P_0 v \cdot \nabla A_0}{\langle v \rangle^2}\right)\right] \\ &=-\frac{1}{4\langle v \rangle^2}\left(\pm{(\frac{v}{\langle v \rangle}\times B)_m+\nabla_m A_0} \right)[P_0, \alpha_m] \\
    &= -\frac{1}{2\langle v \rangle^2}\beta\alpha\cdot\left(\pm{(\frac{v}{\langle v \rangle}\times B)+\nabla A_0} \right) \\
    &\qquad  {+} \frac{i}{2\langle v \rangle^2} \gamma^5 \alpha \cdot \left( v \times \left(\pm{(\frac{v}{\langle v \rangle}\times B)+\nabla A_0} \right)\right).
\end{align*}
A similar reasoning for $\partial_t \Pi_\pm$ yields, using \eqref{eq: explicit formula time derivative Pi},
\begin{align*}
    \left[\Pi_\pm, \partial_t \Pi_\pm\right] &= \left[\frac{1}{2}\pm\frac{P_0}{\langle v \rangle},\mp\frac{1}{2\langle v \rangle}\left(\alpha \cdot \partial_tA - \frac{P_0 v\cdot \partial_t A}{\langle v \rangle^2}\right)\right] \\
    &= -\frac{1}{4\langle v \rangle^2}\partial_tA_m\left[{P_0},\alpha_m  \right]\\
    &= -\frac{1}{2\langle v \rangle^2}\left(\beta \alpha \cdot \partial_t A {-}  i\gamma^5\alpha \cdot (v\times \partial_t A)\right).
\end{align*}
Using $E= \nabla A_0 -\partial_t A$ we have
\begin{equation}
    H_{\text{be}} = -\frac{1}{2\langle v \rangle^2}(\beta\alpha\cdot F^{\mathrm{L}}_{\pm} {-} i\gamma^5 \alpha \cdot (v\times F^{\mathrm{L}}_{\pm})),
\end{equation}
where 
\begin{equation}
    F^{\mathrm{L}}_{\pm} := E  \pm(\frac{v}{\langle v \rangle}\times B),
\end{equation}
is the Lorentz force for electrons ($+$) and positrons ($-$), respectively.
\subsection{Poissonian curvature}
\label{sec:poissonian curvature}

The Poissonian curvature \eqref{eq:poissonian curvature} is given by
\begin{align*}
     H_{\mathrm{pc}}^{\pm}  
    &= \pm\langle v \rangle \Pi_\pm \{\Pi_{\mp},\Pi_{\mp}\} \Pi_{\pm}
\end{align*}
Thus it remains to calculate $\{\Pi_+,\Pi_+\}$, which was done in \eqref{eq: explicit formula Poisson bracket Pi Pi}.
Since $$ \Pi_{\pm}\left[\{ P_0,\frac{1}{\langle v \rangle} \},{P}\right] \Pi_{\pm} = \Pi_{\pm}\left[\{ P_0,\frac{1}{\langle v \rangle} \},\lambda_+\Pi_+ + \lambda_- \Pi_-\right] \Pi_{\pm}= 0,$$
we obtain
\begin{equation}
H^{\pm}_{\mathrm{pc}} = \pm  \frac{1}{2\langle v \rangle}\Pi_{\pm}\left(i\gamma ^5 \alpha  \cdot B\right) \Pi_{\pm}.
\end{equation}
\section{Wigner matrices}

We recall here the theory of matrix-valued Wigner measures. They were treated in \cite{gerard1997homogenization}, based on the treatment of scalar Wigner measures in \cite{lions1993mesures}. \\

Let $\{R^{\varepsilon}\} \subset \textfrak{S}^1(\mathfrak{H})$ be a family of matrix-valued density operators on $\mathfrak{H}= L^2(\mathbb{R}^d, {\mathbb{C}}^m)$, indexed by $\varepsilon$. Let $W^{\varepsilon}$ be the corresponding family of matrix-valued Wigner transforms. Set
\begin{equation}
    \widehat{R}(\xi) \equiv \widehat{R}(\xi,\xi) := \int e^{-i\xi\cdot(x-y)} R(x,y) dx dy.
\end{equation}

\begin{definition}
    A family of measures $\mu^{\varepsilon}$ on $\mathbb{R}^d$, indexed by $\varepsilon \in (0,1]$, is \emph{tight} if,
    \begin{equation}
        \sup_{0<\varepsilon \leq 1}\int_{|x|\geq M} \mu^{\varepsilon}(dx) \rightarrow 0 \quad \text{as } M\rightarrow \infty.
    \end{equation}
\end{definition}  
Some authors define a
family $\{\psi_{\varepsilon} \colon \varepsilon \in (0,1]\}$ as being \emph{compact at infinity} if the measure with density $\mu^{\varepsilon} = |\psi_{\varepsilon}|^2$
(with respect to the Lebesgue measure on $\mathbb{R}^d$) is tight. The term \emph{$\varepsilon$-oscillatory} is used for the tightness of $\varepsilon^{-d}\widehat{\mu}(\varepsilon^{-1}\xi)$.
\begin{theorem}[\cite{gerard1997homogenization, lions1993mesures}]
\label{thm:wigner measures}
    The Wigner transform $W^{\varepsilon}$ has the following properties:
    \begin{enumerate}
        \item The family $\{W^{\varepsilon}\colon 0 < \varepsilon \leq 1\}$ is bounded in $\mathcal{S}'(\mathbb{R}_x^d\times \mathbb{R}^d_{\xi},\mathbb{M}_m(\mathbb{C}))$,  in particular there
exists a subsequence $\varepsilon_k$ such that $W^{\varepsilon_k}$ converges to the bounded matrix-valued Wigner measure ("Wigner matrix") $W=W^{\ast}\geq 0$ in
$\mathcal{S}'(\mathbb{R}_x^d\times \mathbb{R}^d_{\xi},\mathbb{M}_m(\mathbb{C}))$ as $\varepsilon_k\rightarrow 0$. The scalar measure $w:= \tr_{\mathbb{C}^m}(W)$ is a positive and bounded Radon measure on $\mathbb{R}^d_x\times \mathbb{R}^d_{\xi}$. 
\item
If the diagonal $R^{\varepsilon}(x) = \sum_{j\geq 1} \lambda_j^{\varepsilon}\psi_j^{\varepsilon}(x)\overline{\psi^{\varepsilon}_j(x)}^\top$ converges to $\nu$ in the sense of matrix-valued measures on $\mathbb{R}_x^d$, then
\begin{equation}
    \int_{\mathbb{R}^d} W(\cdot,d\xi) \leq \nu.
    \end{equation}
Equality holds if in addtion $\varepsilon^{-d}\widehat{R^{\varepsilon}}(\xi/\varepsilon)$ is tight.
\item We have
\begin{equation}
    \int w(dx,d\xi) \leq \limsup_{\varepsilon\rightarrow 0} \tr(R^{\varepsilon}).
\end{equation}
Equality holds if in addition $R^{\varepsilon}(x)$ and $\varepsilon^{-d}\widehat{R^{\varepsilon}}(\xi/\varepsilon)$ are tight.
\end{enumerate}
\end{theorem}

\section*{Acknowledgement}
N.J.M. and J.M. acknowledge financial support from the Austrian Science Fund (FWF) via the SFB project 10.55776/F65, and the Schrödinger grant 10.55776/J4840.

F.G., N.L. and J.M. acknowledge the hospitality and support of the Wolfgang Pauli Institute Vienna.

J.M. acknowledges the hospitality of Constructor University, Bremen.

C.S. acknowledges the support of the Swiss National Science Foundation through
the NCCR SwissMAP, the SNSF Eccellenza project PCEFP2\_181153, and the Swiss State Secretariat for Research and Innovation through the ERC
Starting Grant project P.530.1016 (AEQUA).


\bibliographystyle{abbrv}
\bibliography{dirac.bib}

\end{document}